\documentclass[final]{siamltex}
\usepackage{animate}
\usepackage{amsmath}
\usepackage{amssymb}
\usepackage{graphics}
\usepackage{graphicx}
\usepackage{textcomp}
\usepackage{mathrsfs}
\usepackage{epstopdf}
\usepackage{array}
\usepackage{cite}
\usepackage[maxfloats=99]{morefloats}
\usepackage{color}
\usepackage{url}
\usepackage{color}
\usepackage{cite}
\usepackage{subcaption}

\usepackage{stmaryrd}
\usepackage{bm}

 \usepackage{pgfplots}
  \pgfplotsset{compat=newest}
  \usetikzlibrary{plotmarks}
  \usetikzlibrary{arrows.meta}
  \usepgfplotslibrary{patchplots}
  \usepackage{grffile}

\renewcommand{\theequation}{\arabic{section}.\arabic{equation}}

\newtheorem{remark}{Remark}[section]
\newtheorem{define}{Definition}[section]

\graphicspath{{fig/}}

\title{Surface Crouzeix-Raviart  element for the Laplace-Beltrami equation}

\author{ Hailong Guo\thanks{School of Mathematics and Statistics,  The University of Melbourne,  Parkville, VIC 3010, Australia   (hailong.guo@unimelb.edu.au). 
This work was partially supported by Andrew Sisson Fund of the University of Melbourne. }}

\begin{document}

\maketitle

%
%
\medskip

\begin{abstract}
This paper is concerned with the nonconforming finite element discretization of geometric partial differential equations. 
In specific, we construct a surface Crouzeix-Raviart element on the linear approximated surface, analogous to a flat surface. 
The optimal error estimations are established even though the presentation of the geometric error.  
By taking the intrinsic viewpoint of manifolds, we introduce a new superconvergent gradient recovery method for the surface Crouzeix-Raviart element using only the information of discretization surface.  
  The potential of serving as an asymptotically exact {\it a posteriori} error estimator is also exploited.  
A series of benchmark numerical examples are presented to validate the theoretical results and numerically demonstrate the superconvergence of the gradient recovery method.

\vskip .7cm
{\bf AMS subject classifications.} \ {Primary 65N30; Secondary 45N08}

\vskip .3cm

{\bf Key words.} \ {Laplace-Beltrami operator, nonconforming, surface finite element method, Crouzeix-Raviart element, gradient recovery, superconvergence.}
\end{abstract}

\section{Introduction}
Numerical methods for approximating partial differential equations (PDEs) with solutions defined on surfaces are of growing interests over
the last decades.  Since the pioneer work of Dziuk \cite{Dz1988},   there is  tremendous development on   finite element 
methods\cite{BHLLM2018, ADSS2015, DMS2013, De2009, DD2007, DE2013, GR2016, LL2017,ORG2009,OS2016, CD2016}.  Fluid equations on manifolds have many important applications 
in  fluidic biomembranes\cite{AD2009, BGN2015}, computer graphics\cite{ETKSD2007, MCPTD2009},  geophysics\cite{PRPV2017, STY2015}. Typically,  numerical simulation of surface Stokes or Navier-Stokes equations is unavoidable.  In the literature, there are several
works on them, for example, see \cite{Fr2018, RV2018, Re2018, BGN2016, OQRY2018}.  It is well known that linear surface element is not a stable pair for surface Stokes equations \cite{BF1991}. 
One can fix it by adding a stabilizing term\cite{RV2018, OQRY2018} or using Taylor-Hood element \cite{Fr2018}.

In the planar domain case,  a simpler way to overcome this difficulty is to use the Crouzeix-Raviart element.  The Crouzeix-Raviart element was firstly proposed by Crouzeix and Raviart in \cite{CR1973} to solve a steady Stokes equation. 
Different from the Courant element, such element is only continuous at edge centers of a triangulation. 
In that sense, it is a nonconforming element.  In addition to being used to construct a simple stable finite element pair for Stoke problems, the method is also proven to be locking free for Lame problems\cite{BS1992}. It can be viewed as a universal element for solids, fluids, and electromagnetic, 
see the recent review paper \cite{Br2014} and the references therein.

Our first purpose is to extend this exotic nonconforming element to a surface setting.  
Compared with the counterpart in the flat space,  there is an additional geometric error due to the discretization of the surface. 
One of the main difficulties is to estimate the nonconforming error. 
The key ingredient of this step is to conduct all the error analysis on the discretized surface instead of on the exact surface.
It should be pointed out that, in general, two triangles sharing a common edge are not on the same plane. 
The standard argument \cite{Br2007, Ci2002} for nonconforming finite element method cannot be applied directly 
and the nonconforming error is coupled together with the geometric error.  By carefully using the geometric approximation properties, 
we show that the geometric error has no impact on the overall convergence results.

Our second purpose is to propose a superconvergent post-processing technique for the surface Crouzeix-Raviart element. 
On the planar domain case, there are several post-processing techniques\cite{CB2002, GZ2015} for the Crouzeix-Raviart element. In particular, Guo and Zhang employed a local least-squares fitting procedure at every edge center to generate a more accurate approximate gradient.   The most straightforward way
of generalizing such idea to a surface setting is to project a local patch onto its tangent plane as in \cite{WCH2010}.  However, there are two barriers to the surface Crouzeix-Raviart element: 
first, it requires the exact normal vectors; second, it requires the edge centers located on the exact surface. Those two difficulties can be alleviated by going back to the original definition of the covariant derivative as in \cite{DG2017}.   In specific, we firstly adopt a least-squares procedure to recover the local parametric map and then employ another least-squares fitting on the parameter domain.  Based on the gradient recovery method, 
we introduce a recovery-type {\it a posteriori} error estimator for the surface Crouzeix-Raviart element. 

The rest of the paper is organized as the follows.  In section 2, we give a brief introduction to some preliminary knowledge on the tangential derivative and an exemplary model problem.    In section 3,  we introduce the discretized surface and present the surface Crouzeix-Raviart element. 
Section 4 is devoted to the analysis of discrete energy error and $L^2$ error on the discrete surface.  In section 5, we propose a superconvergent post-processing technique.  A series of benchmark numerical examples are presented to support our theoretical finding in Section 6.
Some conclusions are drawn in Section 7.

\section{Preliminary}
\subsection{Notation}
In    the  paper, we shall consider  $\Gamma$ is an oriented, connected, $C^{\infty}$ smooth regular surface in $\mathbb{R}^3$ without boundary. 
The sign distance function of $\Gamma$ is denoted by $d(x)$.
 Let $\nabla$ be the standard gradient operator in $\mathbb{R}^3$.
 Then the unit outward-pointing normal vector is 
$n(x) = \nabla d(x)$ and  the Weingarten map is $\mathbf{H}(x) = \nabla n(x)= \nabla^2 d(x)$.

Let $U = \{ x\in \mathbb{R}^3: \text{dist} (x, \Gamma)  \le \delta\}$ be a strip  neighborhood around $\Gamma$ with distance $\delta$
where $\text{dist} (x, \Gamma)$ is the Euclidean distance between $x$ and $\Gamma$.   Assume $\delta$ is small enough such 
that there exists a unique projection $p(x):  U\rightarrow \Gamma$ in the form of 
 \begin{equation}
p(x) = x - d(x) n(p(x)).
\end{equation}

 Let $P = Id - n \otimes n$ be the tangential projection operator where $\otimes$ is the tensor product.  The tangential gradient  
 of a scalar function  $v$ on $\Gamma$ is defined to be 
 \begin{equation}
\nabla_{\Gamma} u = P \nabla u = \nabla u - (\nabla u \cdot n)n.
\end{equation}
For a vector field $w\in \mathbb{R}^3$,  the tangential divergence is 
\begin{equation}
\mbox{div}_{\Gamma} w = \nabla_{\Gamma} \cdot w  = \nabla \cdot w - n^t\nabla w n
\end{equation}
The Laplace-Beltrami operator $\Delta_{\Gamma}$ is just 
 the tangential divergence of the tangential gradient, i.e.
 \begin{equation}
\Delta_{\Gamma} v = \mbox{div}_{\Gamma} \nabla_{\Gamma}v = \Delta v - (\nabla v \cdot n) (\nabla \cdot n) - n^t\nabla^2 vn.
\end{equation}

 Let $\alpha = (\alpha_1, \alpha_2, \alpha_3)$ 
be the $3$-index and $|\alpha| =\sum_{i=1}^3 \alpha_i$ with $\alpha_i$ being a nonnegative integer.   
Let $D^{\alpha}_{\Gamma}$ be the $|\alpha|$th order tangential derivative.
Assume $\omega$ being a subset of $\Gamma$ and $m$ being a nonnegative integer.
The Sobolev space $H^m(\omega)$  on $\omega$ \cite{Wl1987} is defined as  
 \begin{equation}
H^m(\omega) = \left\{ v\in L^2(\omega) | D^{\alpha}_{\Gamma} v\in L^2(\omega), |\alpha| \le m\right\},
\end{equation}
with norm 
\begin{equation}
\|v\|_{H^m(\omega)} = \left( \sum_{\alpha\le m} \|D^{\alpha}_{\Gamma}\|^2_{L^2(\omega)}\right)^{1/2},
\end{equation}
and semi-norm
\begin{equation}
|v|_{H^m(\omega)} = \left( \sum_{\alpha = m} \|D^{\alpha}_{\Gamma}\|^2_{L^2(\omega)}\right)^{1/2}.
\end{equation}

Throughout this article,  we use $x\lesssim y$ to denote $x\le Cy$  where the letter $C$  denotes a generic constant which is independent of h and may not be the same at each occurrence. 

\subsection{Model problem} 
In this paper, we shall consider the following model Laplace-Beltrami equation
\begin{equation}\label{equ:model}
-\Delta_{\Gamma} u  + u= f,
\end{equation}
for a given $f\in L^2(\Gamma)$. 

The variational formulation of \eqref{equ:model} is to find $u\in H^1(\Gamma)$ such that 
\begin{equation}\label{equ:var}
a(u, v) = \ell(v), \quad \forall v \in H^1(\Gamma)
\end{equation}
where 
\begin{equation}\label{equ:bilinear}
a(w, v) = (\nabla_{\Gamma}w, \nabla_{\Gamma}v)+(w,v),
\end{equation}
and 
\begin{equation}\label{equ:lf}
\ell(v) = (f, w),
\end{equation}
with $(\cdot, \cdot)$ being the standard $L^2$ inner product on $\Gamma$. 
The Lax-Milligram theorem implies \eqref{equ:var} has a unique solution and there holds the following 
regularity  \cite{Aubin1982}
\begin{equation} \label{equ:reg}
\|u\|_{H^2(\Gamma)} \lesssim \|f\|_{L^2(\Gamma)}.
\end{equation}

\section{The nonconforming finite element method}
\subsection{Approximate surface}
Suppose $\Gamma_h $ is a polyhedral approximation of $\Gamma$ with planar triangular surface. 
Let $\mathcal{T}_h$ be the associated mesh of $\Gamma_h$ and $h= \max_{T\in \mathcal{T}_h}\mbox{diam}(T) $
be its maximum diameter.   Furthermore,  we assume the mesh $\mathcal{T}_h$ is sharp regular and quasi-uniform 
triangulation \cite{Br2007, BS2008, Ci2002} and all vertices lie on $\Gamma$. 
 Let  $\mathcal{E}_h$ be the set of all edges of triangular faces in $\mathcal{T}_h$.   For any edge $E\in\mathcal{E}_h$, let $m_E$ be the middle point of edge $E$. 
The set of all edge middle points of $\mathcal{T}_h$ is denoted by $\mathcal{M}_h$.   
For any $T\in \mathcal{T}_h$. let $n_h$ be the unit outer normal vector to $\Gamma_h$ on $T$. 
The projection onto the tangent space of $\Gamma_h$ can be defined
\begin{equation}
P_h = Id - n_h\otimes n_h.
\end{equation}
Similarly, for a scalar function $v$ on $\Gamma_h$,  we can define its tangential gradient
as 
\begin{equation}
\nabla_{\Gamma_h}v = P_h \nabla v,
\end{equation}
and the Laplace-Beltrami operator  on $\Gamma_h$  as
\begin{equation}
\Delta_{\Gamma_h}v = \nabla_{\Gamma_h}\cdot \nabla_{\Gamma_h}v_h
\end{equation}

Recall that $p(x)$ is a projection map from $U$ to $\Gamma$.  For any $T\in \mathcal{T}_h$, let $T^l = p(T)$ be the curved triangular face
on $\Gamma$.  Denote the set of all curved triangular faces by  $\mathcal{T}_h^l$ , i.e. $\mathcal{T}_h^l = \left\{ T^l: T\in \mathcal{T}_h \right\}$.
Then $\mathcal{T}_h^l$ forms a conforming triangulation of $\Gamma$ such that
\begin{equation}
\Gamma = \bigcup_{T^l\in \mathcal{T}_h^l} T^l. 
\end{equation}

For any edge $E\in \mathcal{E}_h$,  there exists two triangles $T^+$ and $T^-$ such that $E = \partial T^+ \cap \partial T^-$.
The projection on $T^+$ and $T^-$ are denoted by $P_h^+$ and $P_h^-$.  Also, we use the notation $\nabla_{\Gamma_h}^+ = P_h^+\nabla$
to denote the tangent gradient in $T^+$. Similar notation is adopted in $T^-$. 
The conormal of $E$ to $T^+$, which is denoted by $n_E^+$,  is the unit outward vector of $E$  in the tangent plane of $T$ .
Similarly, let $n_E^-$ be the conormal of $E$ to $T^-$.  Analogously, on the curved edge $E^l = p(E)$, we denote 
its conormals by $n_{E^l}^{\pm}$.  Note that $n_{E^l}^+ = - n_{E^l}^-$. 
The unit outer normals of $\Gamma_h$ on $T^+$ and  $T^-$ are denoted by $n_h^+$ and $n_h^-$, respectively. 
Then it easy see that $n_h^+ \perp n_E^+$ and $n_h^- \perp n_E^-$.  We define the jump of  a function $v_h$ 
across $E$ by 
\begin{equation}\label{equ:jump}
\llbracket v \rrbracket  = \lim_{s\rightarrow 0_+ }\left( v(x-s n_E^+) -  v(x-s n_E^-) \right).
\end{equation}

\subsection{The surface Crouzeix-Raviart finite element method}

The surface Crouzeix-Raviart finite element space on $\mathcal{T}_h$  is defined to be
\begin{equation}
V_h = \left\{ v_h \in L^2(\Gamma_h): v_h|_T \in \mathbb{P}^1(T) \text{ and } v_h \text{ is continuous at } \mathcal{M}_h \right\},
\end{equation}
where $\mathbb{P}_1(T)$  is the set of linear polynomials on $T$.  By the definition of jump  \eqref{equ:jump} and 
the midpoint rule,  a piecewise linear function $v$ is in $V_h$  if and only if 
\begin{equation}
\int_E \llbracket v \rrbracket d\sigma_h = 0.
\end{equation}

To simplify the notation, we firstly define  a discrete  bilinear form $a_h(\cdot, \cdot)$ on $V_h\times V_h$ as 
\begin{equation}
a_h(w_v, v_h) = \sum_{T\in\mathcal{T}_h} \int_T \nabla_{\Gamma_h}w_h\cdot  \nabla_{\Gamma_h}w_h\cdot ds_h 
+ (w_h, v_h)_{\Gamma_h}
\end{equation}
and a linear functional $\ell_h(\cdot) $ on $V_h$ as
\begin{equation}\label{equ:dlf}
\ell_h(v_h) = (f\circ p, v_h)_{\Gamma_h}
\end{equation}
where $(\cdot, \cdot)_{\Gamma_h}$ is the standard $L^2$ inner product of $L^2(\Gamma_h)$.
Then the surface Crouzeix-Raviart  finite element discretization of  the model problem \eqref{equ:model} reads as:  find $u_h \in V_h$ such that 
\begin{equation}\label{equ:fem}
a_h(u_h, v_h) = \ell_h(v_h), \quad \forall v_h \in V_h. 
\end{equation}

 We define an broken $H_1$ semi-norm on $V_h$ as 
 \begin{equation}
|v_h|_{H^1(\Gamma_h; \mathcal{T}_h)}^2 =\sum_{T\in\mathcal{T}_h}\|\nabla_{\Gamma_h}v_h \|_{L^2(T)}^2
\end{equation}
The corresponding discrete energy norm is given by 
 \begin{equation}
\|v\|_{h}^2 = |v_h|_{ H^1(\Gamma_h; \mathcal{T}_h)}^2 + \|v_h\|_{L^2(\Gamma_h)}^2 = a_h(v_h, v_h).
\end{equation}
Then,  it is easy  to show that following Lemma:

\begin{lemma}
 $\|v\|_{h} $  is a norm on $V_h$.   The Lax-Milgram theorem implies the discrete variational problem 
 \eqref{equ:fem} admits a unique solution. 
\end{lemma}

\section{A priori error estimates}\label{sec:err}

\subsection{Lift and extension functions}
To compare the error between the exact solution $u$ defined on $\Gamma$ and the finite element solution $u_h$ 
defined on $\Gamma_h$,  we need to establish connections between the functions defined on $\Gamma$ and $\Gamma_h$.

Following the notation as in \cite{BHLLM2018},   for a function $v$ defined on $\Gamma$,  we extend it to $U$ 
and define the extension $v^e$  by
\begin{equation}\label{equ:extension}
 v^e(x)  = v (p(x)), \quad \forall x\in U.
\end{equation}
Similarly, for a function $v_h$  defined on $\Gamma_h$, we define the lift of $v_h$ onto $\Gamma$ by 
 \begin{equation}
v^{l}(x) = v(\xi(x)), \quad \forall x \in \Gamma,
\end{equation}
where $\xi(x)$ is the unique solution of  
\begin{equation}
x =p(\xi) =  \xi - d(\xi)n(x). 
\end{equation}

Then we build the relationship in gradients of extensions and lifts.   For such propose, we introduce 
the matrix $B = P(x) - d(x)H(x)$. It is easy to check that $B = PB = BP = PBP$.  The following relationship
is proved in \cite{BHLLM2018, DD2007, LL2017}
\begin{equation}\label{equ:gradlift}
\nabla_{\Gamma_h} v^e = P_hB(\nabla_{\Gamma}v)^e.
\end{equation}

Let $ds$ and $ds_h$ be the surface measures of $\Gamma$ and $\Gamma_h$. 
 For any $x\in \Gamma_h$,  \cite{DD2007} shows that 
there exists $\mu_h$ such that  $ds\circ p(x) = \mu_h(x) ds_h(x)$   with 
\begin{equation}
\mu_h(x) = (1-d(x)k_1(x))(1-d(x)k_2(x))n\cdot n_h. 
\end{equation}

Throughout the paper, we assume that $\Gamma_h \subset U$.  In the following, we collect some  geometric 
approximation results which will be used in our proof:
\begin{lemma}
Suppose $\Gamma_h \subset U$ is a polyhedral approximation of $\Gamma$.  Assume the mesh size $h$ is small enough. Then 
the following error estimates hold:
\begin{align}
 \|d\|_{L^{\infty}(T)} &\lesssim h^2, \label{equ:ineqone} \\
 \|1-\mu_h\|_{L^{\infty}(T)} &\lesssim h^2, \label{equ:ineqtwo} \\
 \|n-n_h\|_{L^{\infty}(T)}&\lesssim h,    \label{equ:ineqthree} \\
 \|P-P_h\|_{L^{\infty}(T)}&\lesssim h,    \label{equ:ineqfour} \\
 \|n_{E^l}^{+/-} - P n_{E}^{+/-}\|_{L^{\infty}(T)}&\lesssim h^2, \label{equ:ineqfive}\\
  \|\Delta_{\Gamma_h} u^e - (\Delta_{\Gamma} u)^e\|_{L^{2}(T)}&\lesssim h\|u\|_{H^2(T^l)},\label{equ:ineqsix}
\end{align}
where $|\cdot|$ is the standard Euclidean norm. 
\end{lemma}
\begin{proof}
The inequalities \eqref{equ:ineqone}--\eqref{equ:ineqfour} can be proved using the standard linear interpolation theory.  
Their proof can be found in \cite{DD2007}.  
The last two estimates were proved in \cite{LL2017}.
\end{proof}

\begin{remark}
 For the planar domain case,  it is well known that $n_E^+ = - n_E^-$ and hence $ |n_E^++ n_E^-| =0$.
 But this relationship does not hold any more in the surface setting. 
\end{remark}

To connect the  function defined on the exact surface and its extension on the discrete surface, we need 
the following norm equivalence theorem whose proof can be proved in \cite{Dz1988, LL2017}
\begin{lemma}
Let $T\in \mathcal{T}_h$. If  $v\in H^2(T)$, then  the following results hold:
\begin{align}
 \|v^l\|_{L^2(T^l)} \lesssim & \|v\|_{L^2(T)} \lesssim  \|v^l\|_{L^2(T^l)}, \label{equ:eleml2}\\
  |v^l|_{H^1(T^l)} \lesssim & |v|_{H^1(T)} \lesssim  |v^l|_{H^1(T^l)}, \label{equ:elemh1}\\
& |v|_{H^2(T)} \lesssim  \|v^l\|_{H^2(T^l)}, \label{equ:elemh2}\\
  |v^l|_{H^2(T^l)} \lesssim & \|v\|_{H^2(T)}. \label{equ:elemh22}
\end{align}
\end{lemma}

Also, we need the following norm equivalence results for function defined on the edge of an element \cite{BHLLM2018}
\begin{lemma}
Let $E\in \mathcal{E}_h$. If  $v\in H^1(E)$, then  the following results hold:
\begin{align}
 \|v^l\|_{L^2(E^l)} \lesssim & \|v\|_{L^2(E)} \lesssim  \|v^l\|_{L^2(E^l)}, \label{equ:edgel2}\\
  |v^l|_{H^1(E^l)} \lesssim & |v|_{H^1(E)} \lesssim  |v^l|_{H^1(E^l)}. \label{equ:edgeh1}
\end{align}
\end{lemma}


\subsection{The nonconforming interpolation}
For any $T\in \mathcal{T}_h$, let $\mathcal{E}_T$ be the set of three edges of $T$. 
We define 
the local interpolation operator   $\Pi_T: H^1(T) \rightarrow \mathbb{P}_1(T)$  
by 
\begin{equation}
(\Pi_Tv)(m_E) = \frac{1}{|E|}\int_E vd\sigma_h, \quad \forall v\in H^1(T), \, E\in \mathcal{E}_T,
\end{equation}
where $|E|$ is the length of $E$.   By the midpoint rue, we can show that 
\begin{equation}\label{equ:eqdef}
\int_{E}(\Pi_Tv)d\sigma_h = \int_{E}vd\sigma_h, \quad  E\in \mathcal{E}_T.
\end{equation}

Let $h_T$ be the diameter of $T$. Then the following error estimate holds \cite{CR1973, Br2014}
\begin{equation}\label{equ:localinterr}
\|v - \Pi v\|_{L^2(T)} +h_T |v - \Pi v|_{H^1(T)}  \lesssim h_T^2 |v|_{H^2(T)},
\end{equation}
for any $v\in H^2(T)$.  The global interpolation operator $\Pi_h: H^1(\Gamma_h) \rightarrow V_h$  is defined by 
\begin{equation}
(\Pi_hv)|_T = \Pi_Tv, \quad \forall T\in \mathcal{T}_h. 
\end{equation}

It follows from \eqref{equ:localinterr} that
\begin{equation}\label{equ:interr}
\|v - \Pi_h v\|_{L^2(\Gamma_h)} +h |v - \Pi_h v|_{H^1(\Gamma_h; \mathcal{T}_h)}  \lesssim h^2 |v|_{H^2(\Gamma_h)},
\end{equation}
In particular, let $v = u^e$. Then we have 
\begin{equation}\label{equ:inferror}
 \inf_{v_h\in V_h}\|u^e-v_h\|_h \le |v - \Pi_h v|_{H^1(\Gamma_h; \mathcal{T}_h)} +   \|u- \Pi_h u^e\|_{L^2(\Gamma_h)}  \lesssim h |u^e|_{H^2(\Gamma_h)}.
\end{equation}

\subsection{Energy error estimate}
In this subsection, we establish the error bound in the  discrete energy error. 
Our main tool of the error estimation is the second  Strang  Lemma \cite{Ci2002, BS2008, Br2007}:
\begin{lemma}[The second Strang Lemma]\label{lem:strang}
 Suppose $u$ is the exact solution of \eqref{equ:var}  and $u_h$ is  the finite element solution of \eqref{equ:fem} .
 Then we obtain that 
\begin{equation}
\begin{split}
 || u^e- u_h||_{h} \lesssim & \inf_{v_h\in V_h} \|u^e - v_h\|_{h} + \sup_{w_h\in V_h}\frac{|a_h(u^e, w_h) -(f^e,  w_h)|}{\|w_h\|_{h}} 
 \end{split}
\end{equation}
\end{lemma}
\begin{remark}
We also call  the first term is the approximation error and the second term is the nonconforming consistency error. But different from the planar domain case, 
the second term also involves the geometric error in addition to the classical nonconforming consistency error. 
 We measure the error  using the discrete energy norm on the approximate surface and this is the key part to bound the nonconforming consistency error. 
\end{remark}

We prepare the energy error estimation with some  geometric error estimates.  We begin with the following Lemma:
\begin{lemma}\label{lem:geo}
 Let $u$ be the solution of \eqref{equ:var} and $u^e$ be its extension to $U$ defined by \eqref{equ:extension}. 
 Then we have the following error estimates holds 
\begin{align}
&| (u, w_h^l) - (u^e, w_h)_{\Gamma_h} | \lesssim h^2 \|u\|_{L^2(\Gamma)}\|w_h\|_{L^2(\Gamma_h)}, \label{equ:geoerrl2}\\
&|(f,  w_h^l)- (f^e, w_h)_{\Gamma_h}  | \lesssim h^2 \|f\|_{L^2(\Gamma)}\|w_h\|_{L^2(\Gamma_h)}.\label{equ:geoerrlf}
\end{align}
for any $w_h \in V_h$. 
\end{lemma}
\begin{proof}
We only prove \eqref{equ:geoerrlf} and \eqref{equ:geoerrl2} can be proved similarly.  Applying the change of variable, we have 
 \begin{equation*}
 \begin{split}
  |(f, w_h^l) - (f^e, w_h)_{\Gamma_h}|
\lesssim  |((\mu_h-1) f^e, w_h)_{\Gamma_h}
\lesssim  h^2 \|f^e\|_{L^2(\Gamma_h)} \|w_h\|_{L^2(\Gamma_h)}
\end{split}
\end{equation*}
where we have used the error estimate \eqref{equ:ineqtwo}. 
\end{proof}

Next, we prove a lemma for estimate the error involving two  conomorals of an edge. 

\begin{lemma}\label{lem:geonc}
 Let $u$ be the solution of \eqref{equ:var} and $u^e$ be its extension to $U$ defined by \eqref{equ:extension}. 
 Then we have the following error estimates holds 
 \begin{equation}\label{equ:lemmaone}
\sum_{E\in\mathcal{E}_h} \int_{E} \left(n_E^+ \cdot \nabla_{ \Gamma_h}^+u^e+ n_E^-\cdot \nabla_{\Gamma_h}^-u^e \right)^2d\sigma_h
\le h^3 \|u\|^2_{H^2(\Gamma)}.
 \end{equation}
\end{lemma}
\begin{proof}
 Using the triangle inequality and \eqref{equ:gradlift}, we have 
  \begin{equation}
 \begin{split}
&\int_{ E} \left(n_E^+ \cdot \nabla_{ \Gamma_h}^+u^e+ n_E^-\cdot \nabla_{\Gamma_h}^-u^e \right)^2d\sigma_h\\
=&\int_{ E} \left(n_E^+ \cdot  P_h^+B(\nabla_{ \Gamma}u)^e+ n_E^-\cdot P_h^-B(\nabla_{ \Gamma}u)^e \right)^2d\sigma_h\\
=&\int_{ E} \left(Pn_E^+ \cdot  B(\nabla_{ \Gamma}u)^e+ Pn_E^-\cdot B(\nabla_{ \Gamma}u)^e \right)^2d\sigma_h\\
\lesssim &\int_{ E} \left((Pn_E^+  - n_{E^l}^+)\cdot B(\nabla_{ \Gamma}u)^e \right)^2d\sigma_h  + \\
&\int_{ E} \left((n_{E^l}^- -  Pn_E^-)\cdot B(\nabla_{ \Gamma}u)^e \right)^2d\sigma_h\\
\lesssim & h^4 \int_{ E} \left|(\nabla_{ \Gamma}u)^e \right|^2d\sigma_h\\
\lesssim & h^4 \int_{E^l} \left|\nabla_{ \Gamma}u \right|^2 d\sigma,
\end{split}
\end{equation}
where we have used \eqref{equ:ineqfive} in the second inequality and norm equivalence \eqref{equ:edgeh1} in the last inequality.

Summing over over all $E\in \mathcal{E}_h$ and applying the trace inequality,  we have 
  \begin{equation}
 \begin{split}
&\sum_{E\in\mathcal{E}_h}\int_{ E} \left(n_E^+ \cdot \nabla_{ \Gamma_h}^+u^e+ n_E^-\cdot \nabla_{\Gamma_h}^-u^e \right)^2d\sigma_h\\
\lesssim & h^4 \sum_{E\in\mathcal{E}_h} \int_{E^l} \left|\nabla_{ \Gamma}u \right|^2 d\sigma\\
\lesssim & h^4\sum_{T\in\mathcal{T}_h}\int_{\partial T^l} \left|\nabla_{ \Gamma}u \right|^2 d\sigma\\
\lesssim & h^4\sum_{T\in \mathcal{T}_h} \left( h^{-1} \|\nabla u\|^2_{L^2(T^l)} + h |\nabla u|^2_{H^1(T^l)}\right)\\
\lesssim & h^3  \|u\|^2_{H^2(\Gamma)},
\end{split}
\end{equation}
which completes our proof. 
\end{proof}

In the next Lemma, we estimate the main term in the nonconforming consistency error by using an argument 
analogous to the Crouzeix-Raviart element in planar domain \cite{Br2007}.
\begin{lemma}\label{lem:nc}
 Let $u$ be the solution of \eqref{equ:var} and $u^e$ be its extension to $U$ defined by \eqref{equ:extension}. 
 Then we have the following error estimates holds 
 \begin{equation} \label{equ:nc}
 \begin{split}
 \sum_{E\in\mathcal{E}_h}\int_{ E} n_E^+\cdot \nabla_{\Gamma_h}^+u^e \llbracket w_h \rrbracket d\sigma_h
 \lesssim h|u^e|_{H^2(\Gamma)} |w_h|_{H^1(\Gamma_h; \mathcal{T}_h)}.
\end{split}
 \end{equation}
for any $w_h \in V_h$. 
\end{lemma}
\begin{proof} Let $\Pi_E^0w_h = \frac{1}{|E|} \int_Ew_hd\sigma_h$.  Using the fact $\llbracket \Pi_E^0w_h \rrbracket =0$ and  the Cauchy Schwartz inequality,  we have 
 \begin{equation}\label{equ:ncone}
 \begin{split}
& \sum_{E\in\mathcal{E}_h}\int_{ E} n_E^+\cdot \nabla_{\Gamma_h}^+u^e \llbracket w_h \rrbracket d\sigma_h\\
= & \sum_{E\in\mathcal{E}_h}\int_{ E} n_E^+\cdot \nabla_{\Gamma_h}^+u^e \llbracket w_h -\Pi_E^0w_h  \rrbracket d\sigma_h\\
= & \sum_{E\in\mathcal{E}_h}\int_{ E} n_E^+\cdot \nabla_{\Gamma_h}^+(u^e - \Pi_hu^e) \llbracket w_h -\Pi_E^0w_h  \rrbracket d\sigma_h\\
= & \sum_{E\in\mathcal{E}_h}\left( \int_{ E} |\nabla_{\Gamma_h}^+(u^e - \Pi_hu^e) |^2 d\sigma_h\right)^{1/2} \left( \int_{ E} \llbracket w_h -\Pi_E^0w_h  \rrbracket^2 d\sigma_h \right)^{1/2}
 \end{split}
 \end{equation}
 Arguing similarly using the trace inequality, the Poincare's inequality and  \eqref{equ:localinterr} as in planar domain \cite{Br2007}, we obtain
 \begin{align}
&\int_{ E} |\nabla_{\Gamma_h}^+(u^e - \Pi_hu^e) |^2 d\sigma_h\lesssim h |u|^2_{H^2(T^+)}, \label{equ:nctwo}\\
 &\int_{ E} \llbracket w_h -\Pi_E^0w_h  \rrbracket^2 d\sigma_h \lesssim   h \left(|w_h|^2_{H^1(T^+)} + |w_h|^2_{H^1(T^-)} \right). \label{equ:ncthree}
 \end{align}
 Combing the estimates \eqref{equ:ncone}--\eqref{equ:ncthree} gives \eqref{equ:nc}.
\end{proof}

Now, we are prepared  to prove the nonconforming consistency error:
\begin{lemma}\label{lem:mainnc}
 Let $u$ be the solution of \eqref{equ:var} and $u^e$ be its extension to $U$ defined by \eqref{equ:extension}. 
 Then we have the following error estimates holds 
 \begin{equation}\label{equ:mainnc}
|a_h(u^e, w_h) - (f^e, w_h)_{\Gamma_h}| \lesssim h \|u\|_{H^2(\Gamma)}\|w_h\|_h.
\end{equation}
for any $w_h \in V_h$. 
\end{lemma}
\begin{proof}
 For any $w_h\in V_h$, we notice that 
 \begin{equation}
\begin{split}
   &a_h(u^e, w_h) -  (f^e, w_h)_{\Gamma_h}  = [a_h(u^e, w_h) - (f, w_h^l)] + [(f, w_h^l)  - (f^e, w_h)_{\Gamma_h}].\\
 \end{split}
\end{equation}
Using \eqref{equ:geoerrlf},  the second term can be estimated as 
 \begin{equation}
|(f,  w_h^l)- (f^e, w_h)_{\Gamma_h} | \lesssim h^2 \|f\|_{L^2(\Gamma)}\|w_h\|_{L^2(\Gamma_h)} \lesssim h^2 \|u\|_{H^2(\Gamma)}\|w_h\|_h .
\end{equation}
where we used the fact $f = - \Delta_{\Gamma} u + u$. 

To estimate the first term, we apply the Green's formula and we obtain that  
\begin{equation}
\begin{split}
  &a_h(u^e, w_h) - (f, w_h^l) \\
     = &\sum_{E\in\mathcal{E}_h}\int_{ E}\left( n_E^+ \cdot \nabla_{\Gamma_h}^+u^e  w_h^+ + n_E^- \cdot  \nabla_{\Gamma_h}^- u^e w_h^- \right)d\sigma_h -\\
     & \sum_{T\in\mathcal{T}_h}(\Delta_{\Gamma_h}u^e, w_h)_T + (\Delta_{\Gamma}u, w_h^l) + (u^e, w_h)_{\Gamma_h} - (u, w_h^l)\\
     =& \sum_{E\in\mathcal{E}_h}\int_{E}\left(n_E^+ \cdot \nabla_{ \Gamma_h}^+u^e+ n_E^-\cdot \nabla_{\Gamma_h}^-u^e \right)w_h^-d\sigma_h +  \left [   (u^e, w_h)_{\Gamma_h} - (u, w_h^l) \right]\\
     & \sum_{E\in\mathcal{E}_h}\int_{E} n_E^+\cdot \nabla_{\Gamma_h}^+u^e  \llbracket w_h \rrbracket \sigma_h 
      +  \left[ (\Delta_{\Gamma}u, w_h^l) - \sum_{T\in\mathcal{T}_h}(\Delta_{\Gamma_h}u^e, w_h)_T \right] \\
      = &I_1 + I_2 + I_3 + I_4.
\end{split}
\end{equation}

To estimate $I_1$,  we use Lemma \ref{lem:geonc}, the Cauchy-Schwartz  inequality, and the trace inequality   to get
\begin{equation*}
\begin{split}
|I_1| \le &\left( \sum_{E\in \mathcal{E}_h} \int_{E}\left(n_E^+ \cdot \nabla_{ \Gamma_h}^+u^e+ n_E^-\cdot \nabla_{\Gamma_h}^-u^e \right)^2d\sigma_h\right)^{1/2} \left( \sum_{E\in \mathcal{E}_h} \int_E ( w_h^-)^2d\sigma_h\right)^{1/2}\\
  \lesssim & h^{3/2} \|u\|_{H^2(\Gamma)}\left( h^{-1/2}\|w_h\|_{L^2(\Gamma_h)}+ h^{1/2}|w_h|_{H^1(\Gamma_h; \mathcal{T}_h)}\right) \\
  \lesssim & h\|u\|_{H^2(\Gamma)} \|w_h\|_h.
\end{split}
\end{equation*}

According to Lemma  \ref{lem:geo} and Lemma \ref{lem:geonc}, we have 
\begin{equation*}
\begin{split}
|I_2| + |I_3|   \lesssim & h\|u\|_{H^2(\Gamma)} \|w_h\|_h.
\end{split}
\end{equation*}

Then, we estimate $I_4$.  By the triangle inequality and  the error estimate \eqref{equ:ineqsix} and \eqref{equ:ineqfour}, we have 
\begin{equation*}
\begin{split}
|I_4| = & |\sum_{T\in\mathcal{T}} (\Delta_{\Gamma}u, w_h^l)_{T^l} - \sum_{T\in\mathcal{T}_h}(\Delta_{\Gamma_h}u^e, w_h)_T|\\
\le & \sum_{T\in\mathcal{T}} \left| (\mu_h(\Delta_{\Gamma}u)^e, w_h^l)_{T} - (\Delta_{\Gamma_h}u^e, w_h)_T\right|\\
\le & \sum_{T\in\mathcal{T}} |((\mu_h-1)(\Delta_{\Gamma}u)^e, w_h)_{T}| +    \sum_{T\in\mathcal{T}} |((\Delta_{\Gamma}u)^e -\Delta_{\Gamma_h}u^e , w_h)_{T} |\\
\lesssim & h\|u\|_{H^2(\Gamma)} \|w_h\|_{L^2(\Gamma_h)} .
\end{split}
\end{equation*}

Summing the  above three error estimates , we complete the proof of \eqref{equ:mainnc}.
\end{proof}

With all the previous preparations, we are in perfect position to prove the following energy error estimate
\begin{theorem}\label{thm:eng}
 Let $u$ be the solution of \eqref{equ:var} and $u^e$ be its extension to $U$ defined by \eqref{equ:extension}. 
 Then we have the following error estimates holds 
 \begin{equation}
\|u^e - u_h\|_h \lesssim h\|f\|_{L^2(\Gamma)}. 
\end{equation}
\end{theorem}
\begin{proof} Using Lemma  \ref{lem:mainnc} and the regularity estimate \eqref{equ:reg}, we obtain that 
    \begin{equation}\label{equ:supnc}
 \sup_{w_h \in V_h} \frac{|a_h(u^e, w_h) - \ell(w_h^l)|}{ \|w_h\|_{h}} \lesssim h\|f\|_{L^2(\Gamma)}.
  \end{equation}
  We complete our proof by combining the Strang Lemma \ref{lem:strang} and the estimates \eqref{equ:inferror} and \eqref{equ:supnc}. 
\end{proof}

\subsection{$L_2$ error estimate} In this subsection, we establish a priori error estimate in $L^2$ norm  using the Abuin-Nitsche's trick \cite{Ci2002, BS2008, Br2007}. 
Let $g = u - u^l_h \in L^2(\Gamma)$.  The dual problem is to find $\phi\in H^1(\Gamma)$ such that 
\begin{equation}\label{equ:dual}
a(v, \phi) = (v, g), \quad \forall v\in H^1(\Gamma). 
\end{equation}
Similarly, we have the following regularity result:
\begin{equation}
\|\phi\|_{H^2(\Gamma)} \le \|g\|_{L^2(\Gamma)}. 
\end{equation}

The surface Crouzeix-Raviart element discretization of the dual problem is to find $\phi_h\in V_h$ such that 
\begin{equation}\label{equ:ddual}
a_h(v_h, \phi_h) = (v_h, g^e), \quad \forall v_h\in V_h,
\end{equation}
where $g^e = (u-u^l_h)^e = u^e - u_h$.   By Theorem \ref{thm:eng}, we have the following energy error estimate
 \begin{equation}
\|\phi^e - \phi_h\|_h \lesssim  h\|g\|_{L^2(\Gamma)}. 
\end{equation}

We begin our $L^2$ error estimate with the following Lemma:
\begin{lemma} \label{lem:exact}
 Let $u$ be the solution of \eqref{equ:var} and $\phi$ be the solution of the dual problem \eqref{equ:dual}. Then we have 
 the following error estimate 
\begin{equation}  \label{equ:diserr}
a_h(u^e, \phi^e) - a(u, \phi) \lesssim h^2\|u\|_{H^2(\Gamma)}\|\phi\|_{H^2(\Gamma)}.
\end{equation}
\end{lemma}
\begin{proof}
 \eqref{equ:diserr} can be proved by using the same technique as in continuous linear surface finite element, see \cite{Dz1988}.
\end{proof}

Then we prove a lemma involving global interpolation $\Pi_h$.
\begin{lemma}\label{lem:agh}
 Let $u$ be the solution of \eqref{equ:var} and $\phi$ be the solution of the dual problem \eqref{equ:dual}. Then we have 
 the following error estimate 
\begin{align}
& a_h(u^e, \phi^e - \Pi_h\phi^e) \lesssim h^2\|u\|_{H^2(\Gamma)}\|\phi\|_{H^2(\Gamma)}, \label{equ:proja}\\
&a_h(u^e-  \Pi_h u^e, \phi^e )  \lesssim h^2 \|u\|_{H^2(\Gamma)}\|\phi\|_{H^2(\Gamma)}.\label{equ:projb}
\end{align}
\end{lemma}
\begin{proof}
 We only give a proof of \eqref{equ:proja} and \eqref{equ:projb} can be proved similarly.  To prove \eqref{equ:proja},  we apply the integration by part formula and 
 use \eqref{equ:eqdef} which gives 
 \begin{equation*}
\begin{split}
a_h(u^e, \phi^e - \Pi_h\phi^e) =& \sum_{T\in\mathcal{T}_h} (\nabla_{\Gamma_h}u^e, \nabla_{\Gamma_h}(\phi^e - \Pi_h\phi^e)) + (u^e, \phi^e - \Pi_h\phi^e)\\
= & -(\Delta_{\Gamma_h}u^e,\phi^e - \Pi_h\phi^e) + (u^e, \phi^e - \Pi_h\phi^e)
\end{split}
\end{equation*}
Then \eqref{equ:proja} follows by the Cauchy-Schwartz inequality, the interpolation error estimate \eqref{equ:interr} and the norm equivalence. 
\end{proof}

Using the above Lemma, we can prove the following consistency error estimate:
\begin{lemma}
 Let $u$ be the solution of \eqref{equ:var} and $\phi$ be the solution of the dual problem \eqref{equ:dual}. Then we have 
 the following error estimate 
\begin{align}
& a_h(u^e, \phi^e - \phi_h) - ((f, \phi) - (f^e, \phi_h)_{\Gamma_h}) \lesssim h^2\|u\|_{H^2(\Gamma)}\|\phi\|_{H^2(\Gamma)}, \label{equ:cona}\\
&a_h(u^e - u_h, \phi^e ) - ((u, g) - (u_h, g^e)_{\Gamma_h})   \lesssim h^2 \|u\|_{H^2(\Gamma)}\|\phi\|_{H^2(\Gamma)}.\label{equ:conb}
\end{align}
\end{lemma}
\begin{proof}
To prove \eqref{equ:cona}, we notice that 
 \begin{equation*}
\begin{split}
 &a_h(u^e, \phi^e - \phi_h) - [(f, \phi) - (f^e, \phi_h)_{\Gamma_h}]\\
 =& [a_h(u^e, \Pi_h\phi^e - \phi_h) - (f^e,  \Pi_h\phi^e - \phi_h)_{\Gamma_h}] + a_h(u^e, \phi^e - \Pi_h\phi^e)  - \\
  & [(f, \phi) - (f^e, \phi^e)_{\Gamma_h} ]- (f^e, \phi^e-\Pi\phi^e)_{\Gamma_h}\\
  =& I_1 + I_2 + I_3 + I_4. 
\end{split}
\end{equation*}
We first estimate $I_1$.  Using Lemma \ref{lem:mainnc}, we obtain 
 \begin{equation*}
\begin{split}
|I_1| 
 \lesssim &h\|u\|_{H^2(\Gamma)} \|\Pi_h\phi^e - \phi_h\|_h\\
  \lesssim &h\|u\|_{H^2(\Gamma)} (\|\Pi_h\phi^e - \phi^e\|_h+\|\phi^e - \phi_h\|_h)\\
 \lesssim &h^2\|u\|_{H^2(\Gamma)}\|\phi\|_{H^2(\Gamma)},
 \end{split}
\end{equation*}

According to Lemma \ref{lem:agh} and Lemma \ref{lem:geo}, we have 
 \begin{equation*}
|I_2| + |I_3| \lesssim h^2\|f\|_{L^2(\Gamma)} \|g\|_{L^2(\Gamma)}. 
\end{equation*}

To estimate $I_4$, we use the Cauchy-Schwartz inequality and \eqref{equ:interr} which yields that 
 \begin{equation*}
|I_4| \le \|f^e\|_{L^2(\Gamma_h)} \|\phi^e-\Pi\phi^e\|_h \lesssim h^2 \|f\|_{L^2(\Gamma)}\|g\|_{L^2(\Gamma)}.
\end{equation*}

Summing all the above error estimates concludes the proof \eqref{equ:cona}.  The error estimate \eqref{equ:conb} can be proved in the same way. 
\end{proof}

Now, we are ready to present our error estimate in $L^2$ norm. 
\begin{theorem}\label{thm:l2err}
 Let $u$ be the solution of \eqref{equ:var} and $u^e$ be its extension to $U$ defined by \eqref{equ:extension}. 
 Then we have the following error estimates holds 
 \begin{equation}
\|u^e - u_h\|_{L^2(\Gamma_h)} \le h^2\|f\|_{L^2(\Gamma)}. 
\end{equation}
\end{theorem}
\begin{proof}
Using \eqref{equ:var}, \eqref{equ:fem},  \eqref{equ:dual} and \eqref{equ:ddual}, we have  
 \begin{equation*}
\begin{split}
&\|u^e- u_h\|_{L^2(\Gamma_h)}^2 =  (u^e - u_h, g^e) _{\Gamma_h}\\
 = &(u^e, g^e)_{\Gamma_h} - (u, g) + (u, g) - (u_h, g^e)\\
 =& (u^e, g^e)_{\Gamma_h} - (u, g) + a(u, \phi) - a_h(u_h, \phi_h)\\
  =&[ (u^e, g^e)_{\Gamma_h} - (u, g) ]+ [ a(u, \phi) - a_h(u^e, \phi^e)]  +[ a_h(u^e, \phi^e)  - a_h(u_h, \phi_h)]\\
    =&[ (u^e, g^e)_{\Gamma_h} - (u, g) ]+ [ a(u, \phi) - a_h(u^e, \phi^e)]  +a_h(u^e-u_h, \phi^e-\phi_h) + \\
     & a_h(u_h, \phi^e-\phi_h) + a_h(u^e-u_h, \phi_h)\\
     =&[ (u^e, g^e)_{\Gamma_h} - (u, g) ] -  [ a(u, \phi) - a_h(u^e, \phi^e)]  +a_h(u^e-u_h, \phi^e-\phi_h) - \\
     & [a_h(u^e, \phi^e - \phi_h) - ((f, \phi) - (f^e, \phi_h)_{\Gamma_h})] -\\
     & [a_h(u^e - u_h, \phi^e ) - ((u, g) - (u_h, g^e)_{\Gamma_h})]\\
     =& I_1 + I_2 + I_3 + I_4 + I_5. 
 \end{split}
\end{equation*}

We first estimate $I_1$.  By applying the change of variable and \eqref{equ:ineqtwo}, we have 
 \begin{equation*}
|I_1| = ((1-\mu_h)u^e, g^e)_{\Gamma_h} \lesssim h^2 \|u\|_{L^2(\Gamma_h)}\|g\|_{L^2(\Gamma_h)} \lesssim h^2 \|f\|_{L^2(\Gamma)}\|g\|_{L^2(\Gamma)}.
\end{equation*}

The estimate of $I_2$ is provided by Lemma \ref{lem:exact} which implies 
 \begin{equation*}
|I_2| \lesssim h^2 \|f\|_{L^2(\Gamma)}\|g\|_{L^2(\Gamma)}.
\end{equation*}

To estimate $I_3$, we apply the Cauchy-Schwartz inequality which gives
 \begin{equation*}
|I_3| \lesssim h^2 \|u^e-u_h\|_h \|\phi^e - \phi_h\|_h \lesssim h^2\|f\|_{L^2(\Gamma)} \|g\|_{L^2(\Gamma)}
\end{equation*}

According to Lemma \ref{lem:agh} and Lemma \ref{lem:geo}, we have 
 \begin{equation*}
|I_4| + |I_5| \lesssim h^2\|f\|_{L^2(\Gamma)} \|g\|_{L^2(\Gamma)}. 
\end{equation*}

We complete the proof by combining all the above estimates. 
\end{proof}

\section{Superconvergent post-processing}  
    In this section, we generalize the parametric polynomial preserving recovery \cite{DG2017}
to the surface Crouzeix-Raviart element.

The key idea of parametric polynomial preserving recovery is to take an intrinsic view on a surface.  In that sense, a surface can be understood as a union of locally parametrized patches by Euclidean planar domains \cite{Le2013, dC1992}.  Let $g$ be the metric tensor 
of the surface $\Gamma$ and $\mathbf{r}: \Omega \subset \mathbb{R}^2  \rightarrow S \subset \Gamma$ be a local geometric mapping. 
Then the tangent gradient  operator $\nabla_{\Gamma}$ can be equivalently defined as 
\begin{equation}\label{equ:indef}
(\nabla_{\Gamma}u)\circ \mathbf{r} = \nabla \bar{u} (g\circ \mathbf{r})^{-1} \partial \mathbf{r}. 
\end{equation}
where  $\bar{u} = u \circ \mathbf{r}$ is the pull back of the function $u$ to the local planar parameter domain $\Gamma$, $\partial \mathbf{r}$
is the Jacobian of $ \mathbf{r}$, and 
\begin{equation}\label{equ:label}
g\circ \mathbf{r} = \partial \mathbf{r}(\partial \mathbf{r})^T. 
\end{equation}
Using the relation \eqref{equ:label}, we can rewrite  \eqref{equ:indef} as
\begin{equation}\label{equ:newindef}
(\nabla_{\Gamma}u)\circ \mathbf{r} = \nabla \bar{u}  (\partial\mathbf{r})^\dag, 
\end{equation}
where $(\partial\mathbf{r})^\dag$ denotes the Moore-Penrose inverse of $\partial\mathbf{r}$.
As proved in \cite{DG2017}, the definition of the tangent gradient  \eqref{equ:indef} is invariant under different chosen of regular isomorphic parametrization function $\mathbf{r}$.

Then our goal is to use this intrinsic definition of the tangent gradient to propose a new gradient recovery method for the surface Crouzeix-Raviart element.  Different from the linear surface element,  the degrees of freedom of the surface Crouzeix-Raviart element are located on the edge midpoints of the approximate surface triangle instead of their vertices.  We follow the idea of the gradient recovery method for the Crouzeix-Raviart element in \cite{GZ2015} and define the gradient recovery operator $G_h: V_h \rightarrow V_h$.   Given a   finite element function $u_h\in V_h$,  we only need to define $(G_hu_h)(x_i)$ for all  $x_i \in \mathcal{M}_h$.

For any $x_i  = m_{E_i}\in \mathcal{M}_h$ and $n\in \mathbb{N}$,    define the union of elements around $x_i$ in the first $n$ layers as follows
\begin{equation}
L(x_i, n) = \bigcup\left\{ T:  T\in \mathcal{T}_h \text{ and } T\cap L(x_i, n-1) \in \mathcal{E}_h \right\},
\end{equation}
with  $L(x_i, 0) = \left\{E_i \right\}$.   Let $\Omega_i = L(x_i, n_i)$  with $n_i$ being  the smallest integer such that  $\Omega_i$ satisfies   the rank condition (see \cite{ZN2005}) in the following sense:

\begin{define}
 The local element patch is said to satisfy the rank condition i if it admits a unique least-squares fitted polynomial   in  \eqref{equ:appsurf} and \eqref{equ:appfem}.
\end{define}


To construct the recovered gradient at the given midpoint $x_i$, we first choice a vector $\phi_i^3$ to be the normal vector of the local coordinate system.  For the sake of  simplicity, we choose  $\phi_3^i = (n_{E_i}^+ + n_{E_i}^-)/2$.   Then we construct a local parameter domain orthogonal to $\phi^{i}_3$.  Select $x_i$ as the original of $\Omega_i$ and choose $(\phi_1^i, \phi_2^i)$ as the normal basis of $\Omega_i$.   We project all the midpoints  in $L_i$ onto the parameter domain $\Omega_i$ and the projections is denoted by $\xi_{i_j}$, for $j=0, \cdots, n_i$.

Then, we reconstruct the local approximation surface $S_i$ over $\Omega_i$.   As in \cite{DG2017},  the approximate surface $S_i$ can be approximated by 
graph of  a quadratic function on $\Omega_i$.  That is $S_i = \tilde{\mathbf{r}}_{h,i}(\Omega_i)  = \cup_{\xi\in\Omega_i}(\xi, s_i(\xi))$, where 
\begin{equation}\label{equ:appsurf}
s_i = \arg \min_{s\in \mathbb{P}_2(\Omega_i)} \sum_{j = 1}^{n_i} |s(\xi_{i_j}) - <x_{i_j}, \phi_3^i>|^2,
\end{equation}
where $<\cdot, \cdot>$ means the Euclidean inner product in $\mathbb{R}^3$.

Our next step is to reconstruct a more accurate gradient for $\nabla \bar{u}_h$ on the parameter domain $\Omega_i$.  To do this, we use 
$\xi_{i_j}$ as sampling points and fit a quadratic polynomial $p_i(\xi)$ over $\Omega_i$ in the least-squares sense
\begin{equation}\label{equ:appfem}
p_i = \arg\min_{p\in \mathbb{P}_2(\Omega_i)} \sum_{j=0}^{n_i}|p(\xi_{i_j}) - u_h(x_{i_j})|^2. 
\end{equation}

Calculate the partial derivatives of both the polynomial approximated surface function in \eqref{equ:appsurf} and the approximated polynomial function of FEM solution in \eqref{equ:appfem}, 
then we can approximate the tangent gradient which is given in \eqref{equ:newindef} as
\begin{equation}
(G_hu_h)(x_i) = (\partial_1 p_i(0, 0), \partial_2 p_i(0, 0)) 
\begin{pmatrix}
 1 &0  &\partial_1s_i(0, 0) \\
  0 &1 &\partial_2s_i(0, 0) \\
\end{pmatrix}^\dag (\phi_1^i, \phi_2^i, \phi_3^i)^T. 
\end{equation}
To multiply with the orthonormal basis $(\phi_1^i, \phi_2^i, \phi_3^i)$  is because we have to unify the coordinates from local ones to a global one.

Let $\{\chi_i(x_i)\}_{x_i\in\mathcal{M}_h}$ be the nodal basis functions of the surface Crouzeix-Raviart element. Then recovered gradient on the whole domain is 
\begin{equation}
G_hu_h = \sum_{x_i\in\mathcal{M}_h}(G_hu_h)(x_i)\chi_{x_i}(x), \quad \forall x\in \Gamma_h. 
\end{equation}

As a direct application of the gradient recovery method,   we  naturally define a recovery-type {\it a posteriori} error estimator for the surface Crouzeix-Raviart element. 
The local {\it a posteriori} error estimator on each element $T$ is defined as 
\begin{equation}\label{equ:localind}
\eta_{h,T}  =  \|G_h u_h - \nabla_{\Gamma_h} u_h\|_{L^2(T)},
\end{equation}
and the global error estimator as
\begin{equation}\label{equ:ind}
\eta_h = \left( \sum_{T\in\mathcal{T}_h} \eta_{h,T}^2 \right)^{1/2}. 
\end{equation}

\section{Numerical Experiments}
In this section, we present several numerical examples to validate the theoretical results and  test the performance of the recovery-based {\it a posteriori} error estimator. 

To generate an initial mesh on a general surface, we adopt the three-dimensional surface mesh generation module of the Computational Geometry Algorithms Library \cite{cgal}.  Meshes on finer levels are generated by firstly using uniform refinement for the first two numerical examples or the newest bisection \cite{Ch2008}  refinement for the other two numerical examples and then projecting them on to the surface.  In general case, there is no explicit projection map available. We will adopt the first order approximation of projection map as given in \cite{DD2007}.  Therefore the vertices of the meshes are not located on the exact manifold but within a distance of $\mathcal{O}(h^2)$ in our test except for the third numerical example.

For the sake of simplifying the notation, we introduce the following notation for errors
\begin{align*}
&e: =\|u^e - u_h\|_{L^2(\Gamma_h)}, \quad  De: =|u^e - u_h|_{H^1(\Gamma_h; \mathcal{T}_h)},\\ 
&D_ie: =\|\Pi_hu^e - u_h\|_{H^1(\Gamma_h; \mathcal{T}_h)},\quad  D_re: =\|\nabla u^e - G_h u_h\|_{L^2(\Gamma_h)}.
\end{align*}

In the following tables, all convergence rates are listed in term of the degree of freedom(DOF). Noticing $\text{DOF} \approx 1/h^2$ the corresponding convergence rates in term of
the mesh size $h$ is double of what we present in the tables.

\subsection{Numerical example 1} 
In this example,  we consider the model problem \eqref{equ:model} on a general surface firstly introduced by  Dziuk in  \cite{Dz1988}. 
Figure \ref{fig:dmesh} show  the discretized surface and its initial mesh.  The exact solution solution $u(x) = x_1x_2$ and the right hand side 
function $f$ can be computed from $u$.

\begin{figure}[ht]
    \centering
    \includegraphics[width=0.4\textwidth]{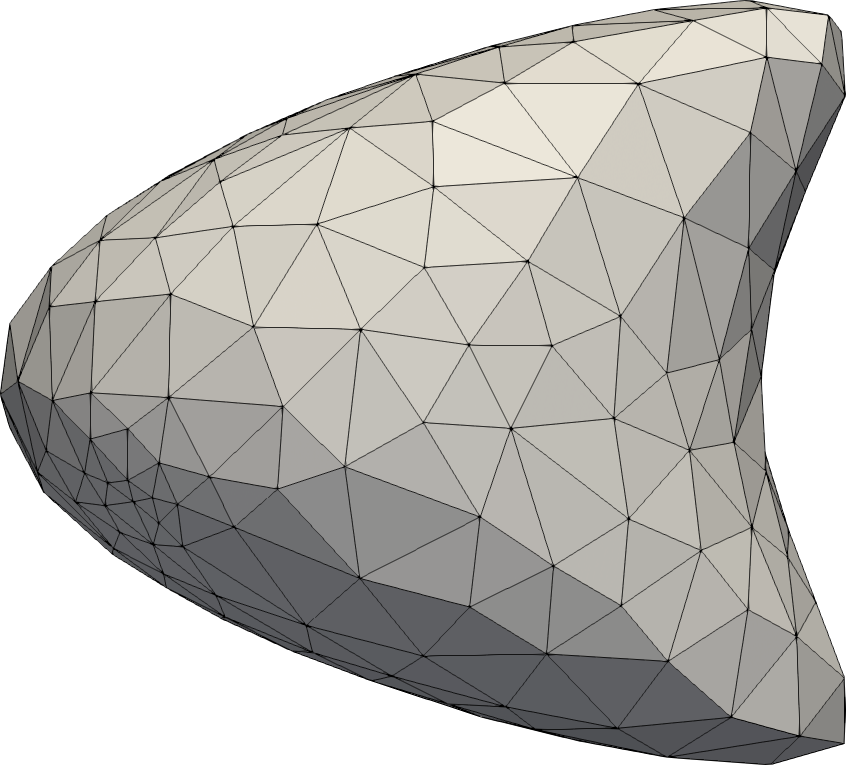}
    \caption{Initial mesh for Dzuik surface}
\label{fig:dmesh}
\end{figure}

We report the numerical results in  Table \ref{tab:dziuk}.   
As predict by the Theorem \ref{thm:l2err} and Theorem \ref{thm:eng},  the $L_2$ error converges at a rate of $\mathcal{O}(h^2)$ 
and the discrete $H^1$ semi-error converges at a rate of $\mathcal{O}(h)$.  Concerning the error between the finite element gradient and the gradient of the interpolation of the exact solution,   $\mathcal{O}(h)$ convergence can be observed.  It means that there is no supercloseness between the finite element gradient and the gradient of the interpolation of the exact solution, which is similar to the numerical results in planar domain \cite{GZ2015}. 
Even though in that case, we can observe $\mathcal{O}(h^{1.9})$ superconvergence for the recovered gradient.

\begin{table}[htb!]
\centering
\footnotesize
\caption{Numerical results for numerical example 1}\label{tab:dziuk}
\begin{tabular}{|c|c|c|c|c|c|c|c|c|c|c|}
\hline 
Dof&$e$&order&$De$&order&$D_ie$&Order&$D_re$&Order\\ \hline\hline 
243&3.70e-02&--&6.95e-01&--&2.10e-01&0.00&3.20e-01&--\\ \hline
966&8.51e-03&1.07&3.66e-01&0.46&1.08e-01&0.48&1.07e-01&0.79\\ \hline
3858&2.19e-03&0.98&1.86e-01&0.49&5.49e-02&0.49&3.06e-02&0.90\\ \hline
15426&5.53e-04&0.99&9.37e-02&0.50&2.77e-02&0.50&8.28e-03&0.94\\ \hline
61698&1.39e-04&1.00&4.69e-02&0.50&1.39e-02&0.50&2.23e-03&0.95\\ \hline
246786&3.47e-05&1.00&2.35e-02&0.50&6.93e-03&0.50&6.15e-04&0.93\\ \hline
\end{tabular}
\end{table}

\subsection{Numerical example 2}
Our second example to consider a general surface with high curvature part as in \cite{DG2017, DE2013}.  
The discretized surface with the initial mesh 
was plotted  in Figure \ref{fig:emesh}. 
We choose $f$ to fit the exact solution $u(x) = x_1x_2$.

\begin{figure}[ht]
    \centering
    \includegraphics[width=0.7\textwidth]{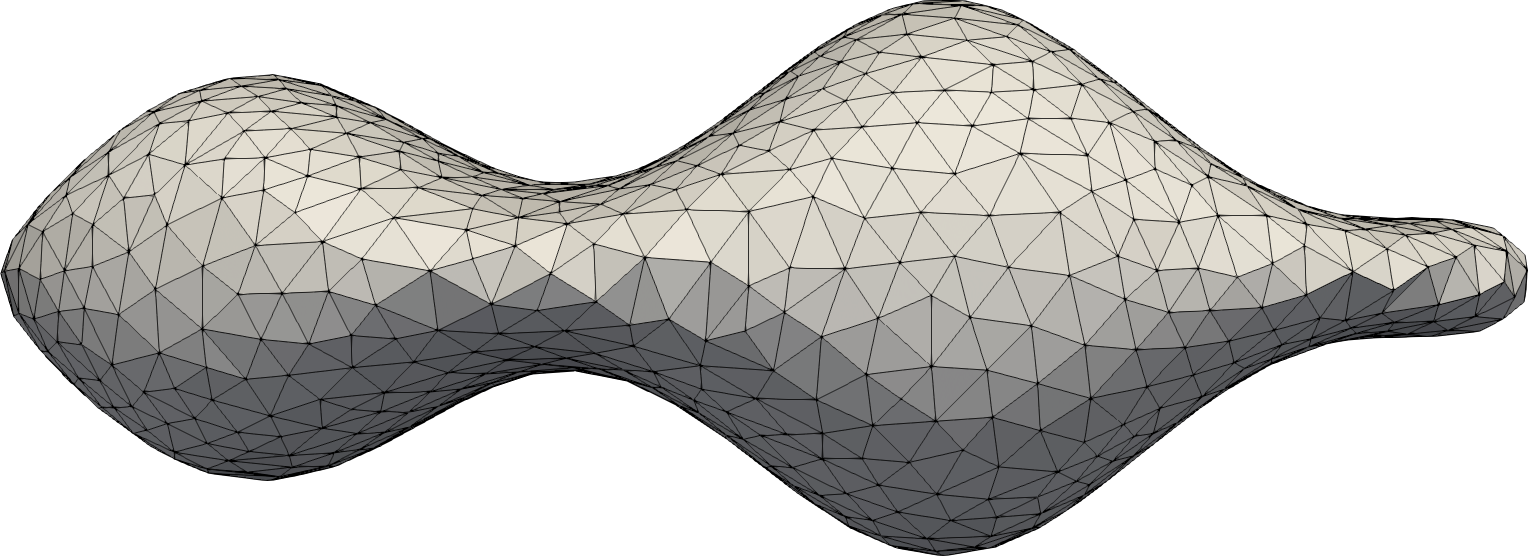}
    \caption{Initial mesh for Elliot surface}
\label{fig:emesh}
\end{figure}

In Table \ref{tab:ell}, we list the history of numerical errors.   We can observe the same optimal convergence results in $L^2$ norm and discrete $H^1$ semi-norm which matches well with the established theoretic results in Section \ref{sec:err}.  Similar to the previous example,  $\mathcal{O}(h^2)$ can be observed even though there is no supercloseness result.

\begin{table}[htb!]
\centering
\footnotesize
\caption{Numerical results for numerical example 2}\label{tab:ell}
\begin{tabular}{|c|c|c|c|c|c|c|c|c|c|c|}
\hline 
Dof&$e$&order&$De$&order&$D_ie$&Order&$D_re$&Order\\ \hline\hline 
1153&8.82e+00&--&2.92e+01&--&2.50e+01&--&1.46e+01&--\\ \hline
4606&5.25e-02&3.70&4.24e-01&3.06&2.26e-01&3.40&2.37e-01&2.98\\ \hline
18418&3.92e-03&1.87&1.86e-01&0.60&5.37e-02&1.04&7.19e-02&0.86\\ \hline
73666&1.08e-03&0.93&9.28e-02&0.50&2.65e-02&0.51&1.98e-02&0.93\\ \hline
294658&2.54e-04&1.05&4.64e-02&0.50&1.32e-02&0.50&5.11e-03&0.98\\ \hline
1178626&6.23e-05&1.01&2.32e-02&0.50&6.59e-03&0.50&1.31e-03&0.98\\ \hline
\end{tabular}
\end{table}

\subsection{Numerical example 3} 
In all the previous numerical examples, the exact solutions are smooth.  In this example, we consider a benchmark problem on the  unit sphere surface with a singular solution.  The solution and the source term in spherical coordinates are given by
\begin{equation*}
u = \sin^{\lambda}(\theta)\sin(\phi), \quad f = (2+\lambda^2 +\lambda)\sin^{\lambda}(\theta)\sin(\phi) + (1-\lambda^2)\sin^{\lambda-2}(\theta)\sin(\phi).
\end{equation*}
It is easy to  show that $u\in H^{1+\lambda}(\Gamma)$. When $\lambda<1$, the solution $u$ has two singularities at north and south poles.

\begin{figure}[h]
   \begin{minipage}[b]{.5\linewidth}
    \centering
   \includegraphics[width=0.8\textwidth]{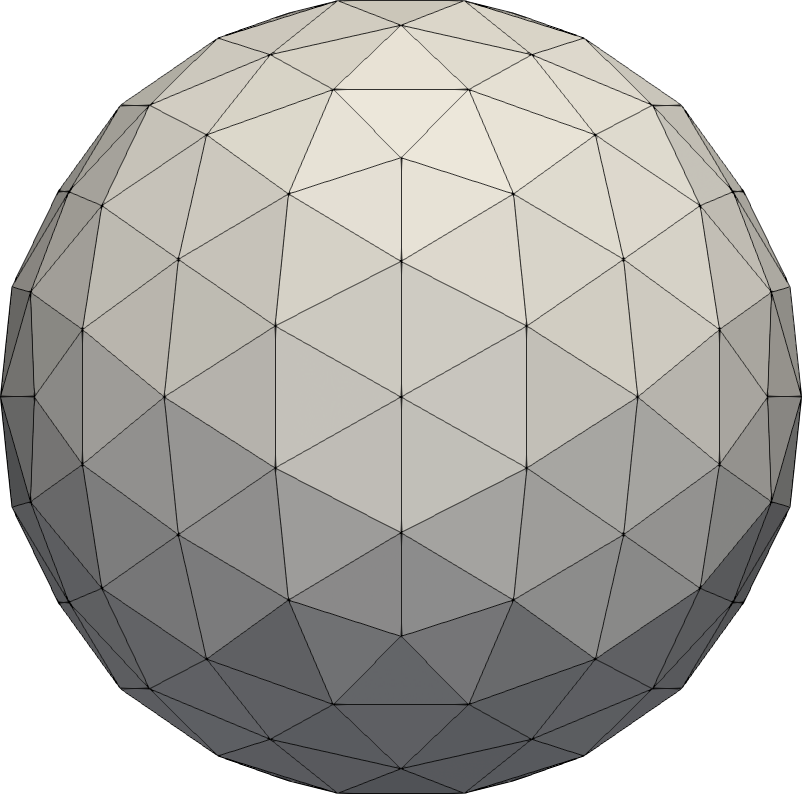}
   \subcaption{} \label{fig:sphere_init}
   \end{minipage}%
   \begin{minipage}[b]{.5\linewidth}
    \centering
   \includegraphics[width=0.8\textwidth]{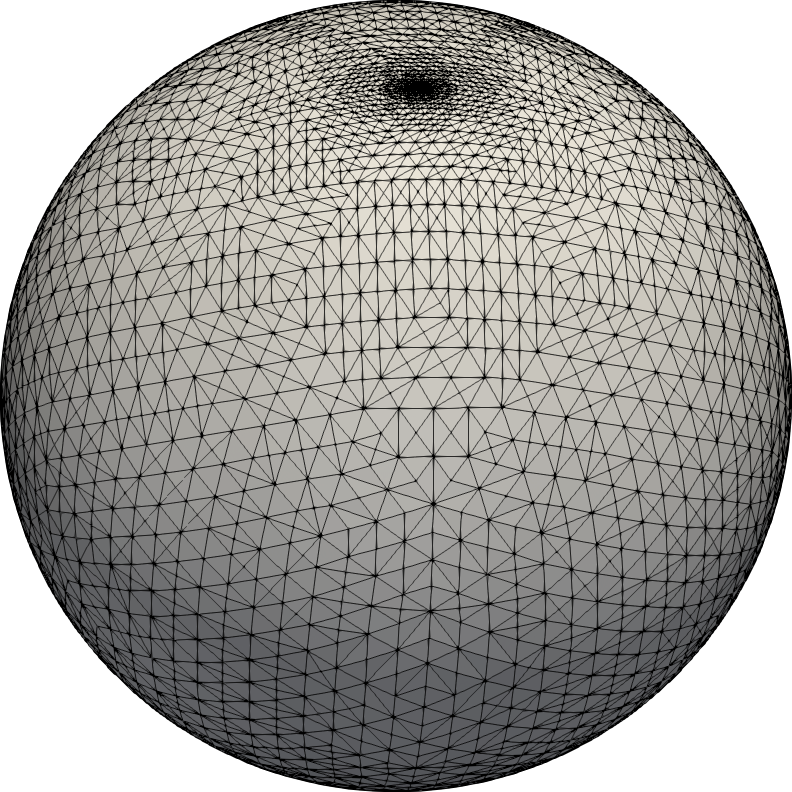}
   \subcaption{}\label{fig:sphere_adaptive}
   \end{minipage}
\caption{Meshes for numerical example 3:  (a)  Initial mesh; (b) Adaptively refined mesh.}
  \label{fig:sphere_mesh}
\end{figure}

To resolve the singularity,  we apply the adaptive finite element method with the recovery based {\it a posteriori} error estimator  \eqref{equ:localind}. 
The initial mesh is icosphere mesh as plotted in Figure \ref{fig:sphere_init}.  Figure \ref{fig:sphere_adaptive} plot the adaptive refined meshes after 14 adaptive refinements.  It obvious that the refinement is  mainly concentrated on the two singular points.  We plot the  errors in Figure \ref{fig:sphere_err}.  The $L^2$ error and discrete $H^1$ semi-error both converges optimally.  The recovered gradient superconverges to the exact gradient at a rate of $\mathcal{O}(h^{1.6})$.

\begin{figure}[h]
   \begin{minipage}[b]{.5\linewidth}
    \centering
   \includegraphics[width=0.8\textwidth]{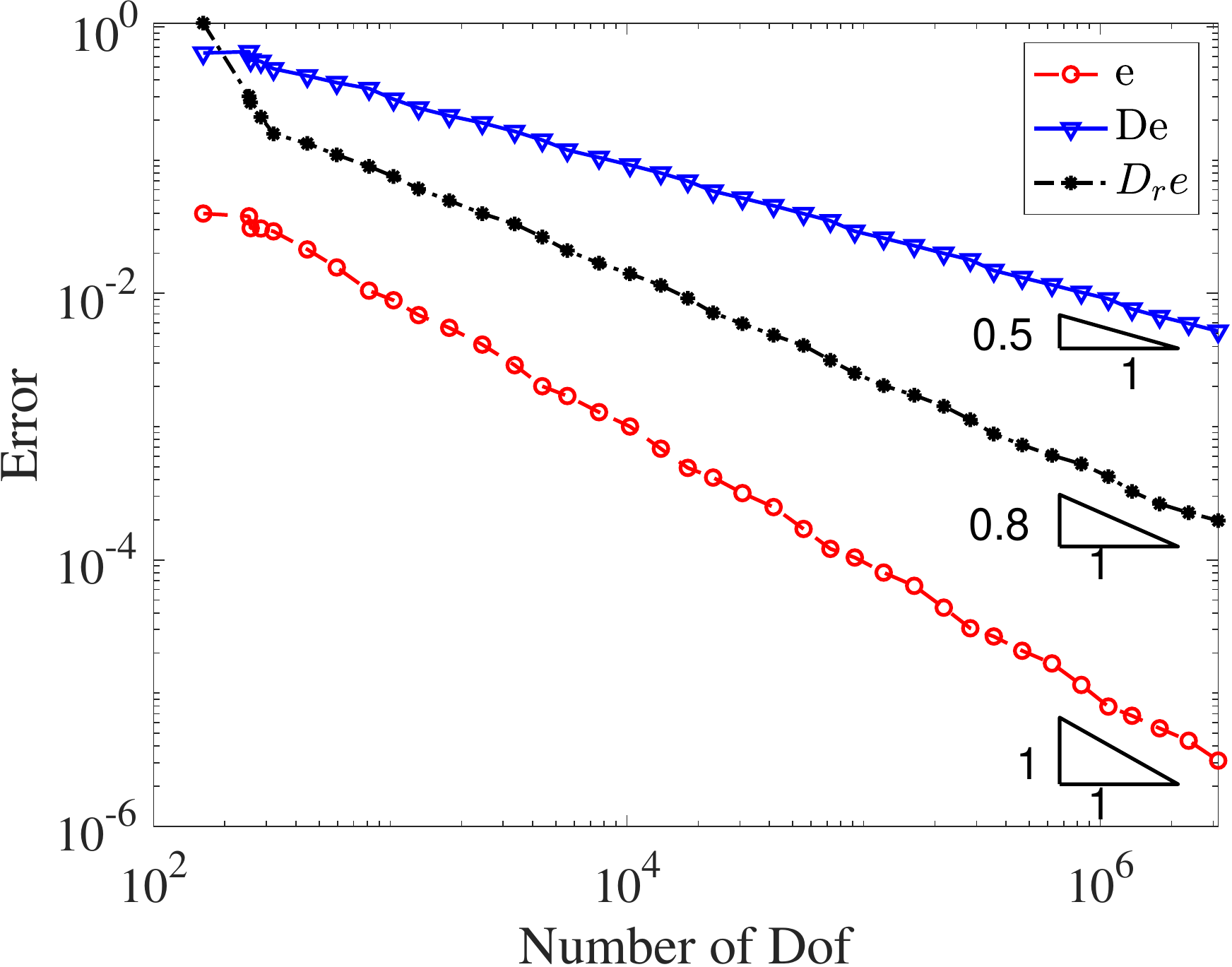}
   \subcaption{} \label{fig:sphere_err}
   \end{minipage}%
   \begin{minipage}[b]{.5\linewidth}
    \centering
   \includegraphics[width=0.8\textwidth]{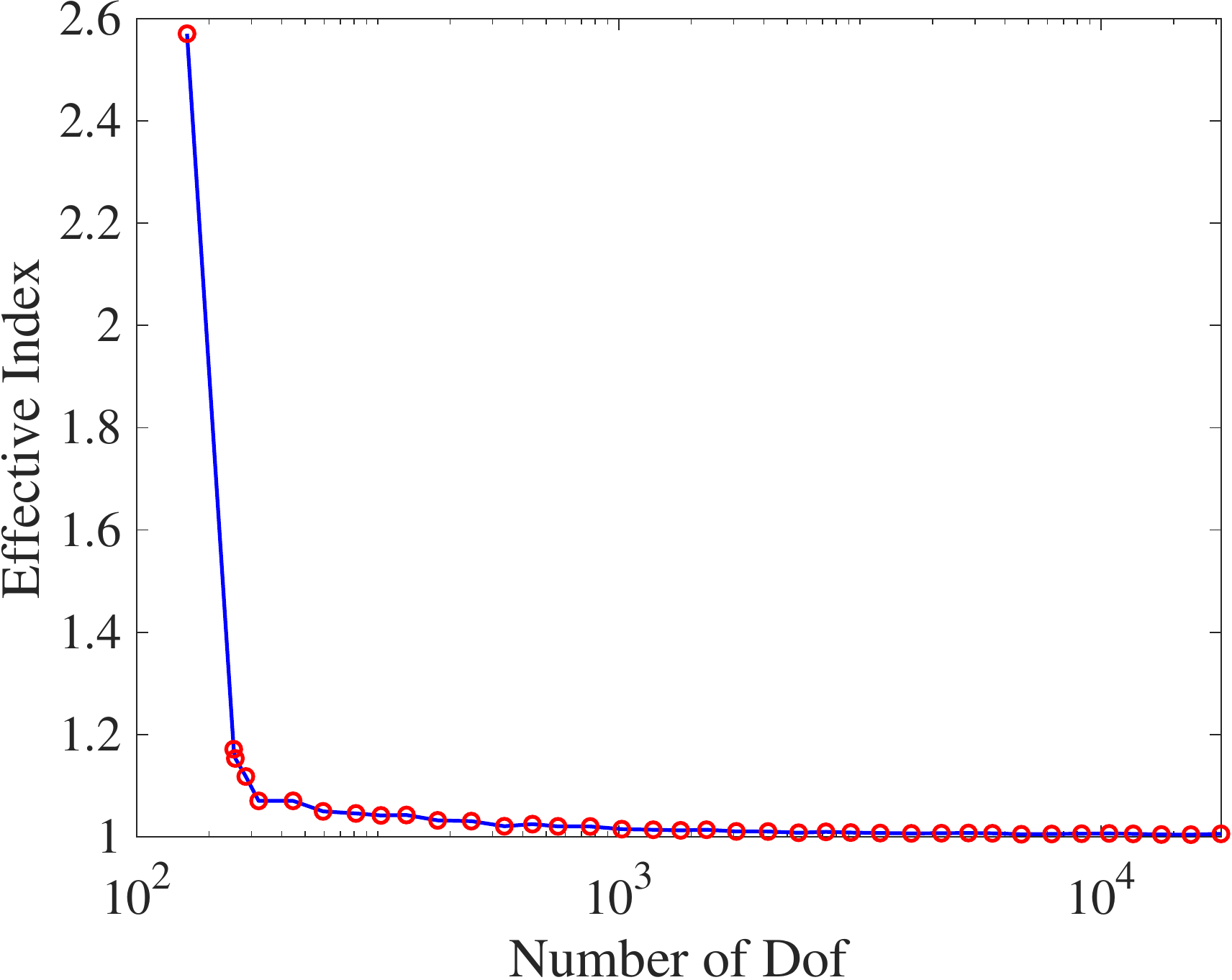}
   \subcaption{}\label{fig:sphere_eff}
   \end{minipage}
\caption{Meshes for numerical example3:  (a)  Initial mesh; (b) Adaptively refined mesh.}
  \label{fig:sphere}
\end{figure}

To quantify   the performance of our new recovery-based \textit{a posterior} error estimator for the Laplace-Beltrami problem, 
the effectivity index $\kappa$ is used to measure the quality of
an error estimator \cite{AO2000,Ba2001}, which  is defined by the ratio between the estimated error and the  exact error

\begin{equation}
    \kappa = \frac{
   \| u_h - \nabla u_h\|_{L^2(\Gamma}}
    {| u - u_h|_{H^1(\Gamma_h; \mathcal{T}_h)}}
    \label{equ:effect}
\end{equation}
The effectivity index is plotted in  Figure \ref{fig:sphere_eff} .
We see that $\kappa$ converges asymptotically to $1$ 
which  indicates the posteriori error estimator 
\eqref{equ:localind}
is asymptotically exact.

\begin{figure}[h]
   \begin{minipage}[b]{.5\linewidth}
    \centering
   \includegraphics[width=0.8\textwidth]{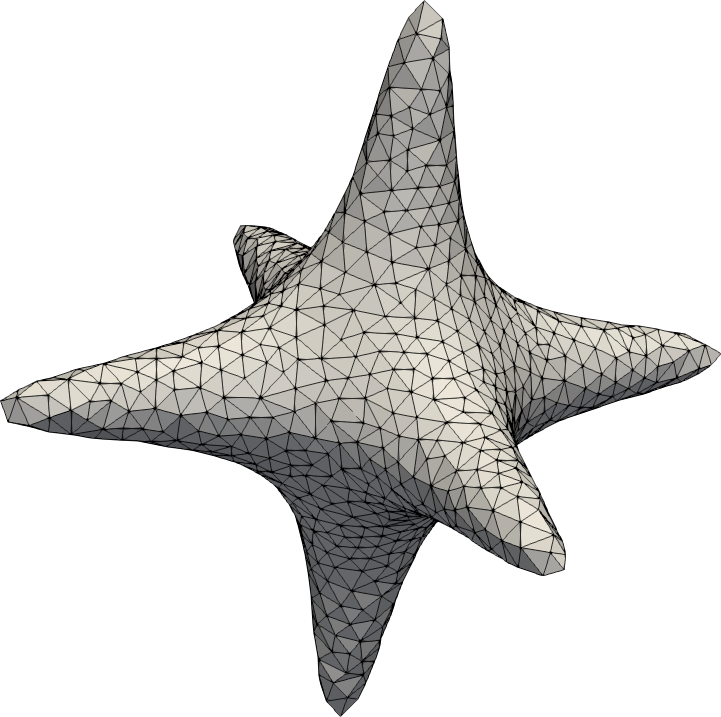}
   \subcaption{} \label{fig:stern_init}
   \end{minipage}%
   \begin{minipage}[b]{.5\linewidth}
    \centering
   \includegraphics[width=0.8\textwidth]{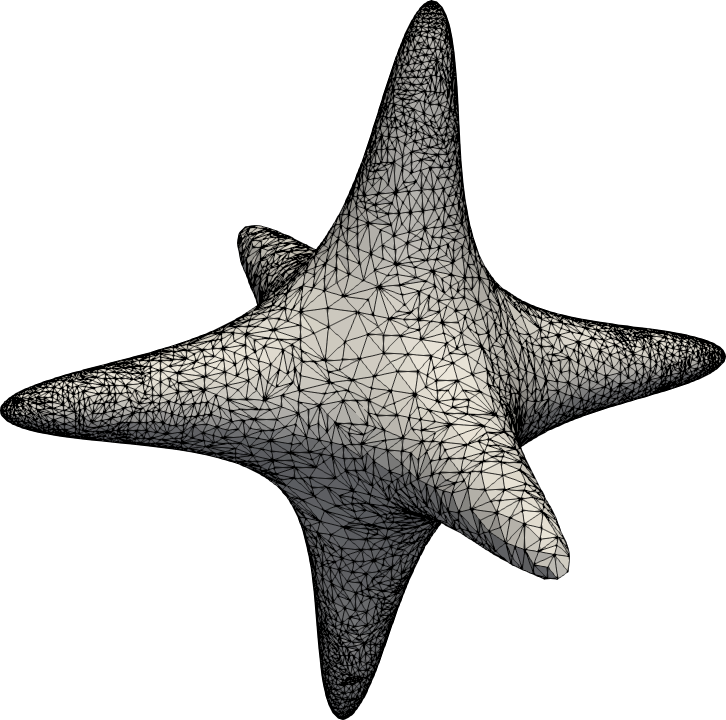}
   \subcaption{}\label{fig:stern_adaptive}
   \end{minipage}
\caption{Meshes for numerical example 4:  (a)  Initial mesh; (b) Adaptively refined mesh.}
  \label{fig:stern_mesh}
\end{figure}

\subsection{Numerical example 4}
This example is taken from \cite{DM2016}.  The surface is the zero level of the following level set function 
\begin{equation*}
\phi(x) = 400(x_1^2x_2^2 + x_2^2x_3^2 + x_1^2x_3^2) - (1-x_1^2-x_2^2-x_3^2)^3-40.
\end{equation*}
The discretized surface on the initial mesh is shown in Figure \ref{fig:stern_init}. What can be clearly seen in this figure is the high curvature parts. The initial mesh fails to resolve them.   In contrast, the high curvature parts are well captured by the adaptively refined mesh as plotted in Figure \ref{fig:stern_adaptive}.

\begin{figure}[h]
   \begin{minipage}[b]{.5\linewidth}
    \centering
   \includegraphics[width=0.8\textwidth]{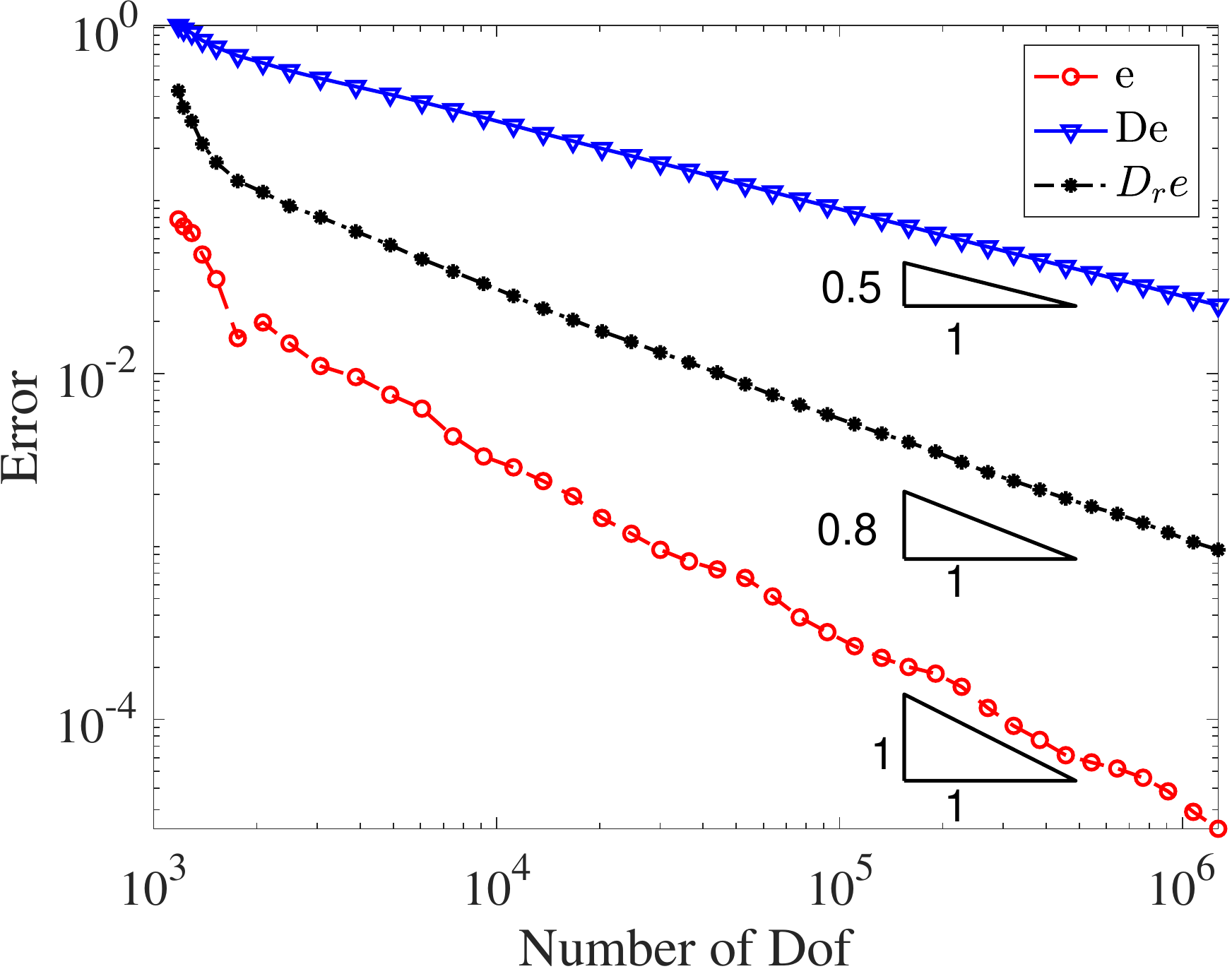}
   \subcaption{} \label{fig:stern_err}
   \end{minipage}%
   \begin{minipage}[b]{.5\linewidth}
    \centering
   \includegraphics[width=0.8\textwidth]{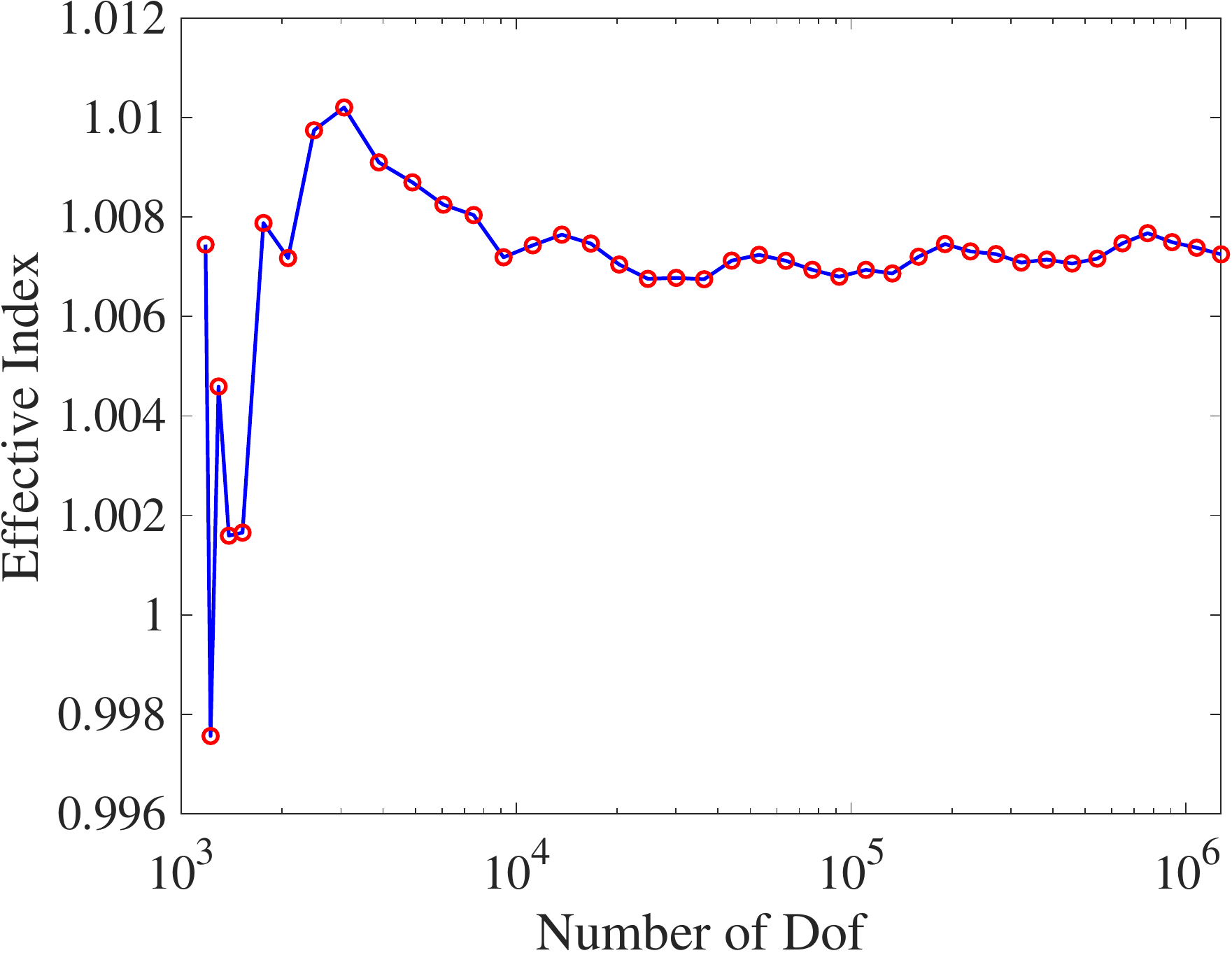}
   \subcaption{}\label{fig:stern_eff}
   \end{minipage}
\caption{Meshes for numerical example 4:  (a)  Initial mesh; (b) Adaptively refined mesh.}
  \label{fig:stern}
\end{figure}

Figure \ref{fig:stern_err} plot errors in term of degrees of freedom.    The figure shows the optimal decay of the $L^2$ error and discrete $H^1$ semi-error.   In addition, we observe that the recovered gradient error superconverges at a rate of $\mathcal{O}(h^{1.6})$.   In Figure \ref{fig:stern_eff}, we graph the effectivity index $\kappa$. Figure 2  reveals that the effectivity index is close to $1$ after several refinements.  It illustrates that the recovery-type {\it a posteriori} error estimator \eqref{equ:localind} is asymptotically exact.

\section{Conclusion}
In this paper, we have introduced and analyzed the Crouzeix-Raviart element on a surface setting.  The surface  Crouzeix-Raviart element
is a nonconforming element in the sense that it is only continuous at edge centers.  The optimal convergence theory has also been established using a delicate argument.  In addition,  we have proposed a superconvergent gradient recovery for the surface Crouzeix-Raviart element.  The proposed post-processing procedure is numerical proven to be able to provide a more accurate approximate gradient and asymptotically exact {\it a posteriori} error estimator. 

Ongoing research topics include using a residue-type a posterior error estimate to conduct medius error analysis \cite{Br2014} and applying it to investigate surface Stokes problems and surface Naiver-Stokes problems. 
\bibliographystyle{siam}
\bibliography{mybibfile}

\begin{thebibliography}{10}

\bibitem{AO2000}
{\sc M.~Ainsworth and J.~T. Oden}, {\em A posteriori error estimation in finite
  element analysis}, Pure and Applied Mathematics (New York),
  Wiley-Interscience [John Wiley \& Sons], New York, 2000.

\bibitem{ADSS2015}
{\sc P.~F. Antonietti, A.~Dedner, P.~Madhavan, S.~Stangalino, B.~Stinner, and
  M.~Verani}, {\em High order discontinuous {G}alerkin methods for elliptic
  problems on surfaces}, SIAM J. Numer. Anal., 53 (2015), pp.~1145--1171.

\bibitem{AD2009}
{\sc M.~Arroyo and A.~DeSimone}, {\em Relaxation dynamics of fluid membranes},
  Phys. Rev. E, 79 (2009), p.~031915.

\bibitem{Aubin1982}
{\sc T.~Aubin}, {\em Best constants in the {S}obolev imbedding theorem: the
  {Y}amabe problem}, in Seminar on {D}ifferential {G}eometry, vol.~102 of Ann.
  of Math. Stud., Princeton Univ. Press, Princeton, N.J., 1982, pp.~173--184.

\bibitem{Ba2001}
{\sc I.~Babu{\v{s}}ka and T.~Strouboulis}, {\em The finite element method and
  its reliability}, Numerical Mathematics and Scientific Computation, The
  Clarendon Press, Oxford University Press, New York, 2001.

\bibitem{BGN2015}
{\sc J.~W. Barrett, H.~Garcke, and R.~N\"urnberg}, {\em Numerical computations
  of the dynamics of fluidic membranes and vesicles}, Phys. Rev. E, 92 (2015),
  p.~052704.

\bibitem{BGN2016}
{\sc J.~W. Barrett, H.~Garcke, and R.~N{\"u}rnberg}, {\em A stable numerical
  method for the dynamics of fluidic membranes}, Numerische Mathematik, 134
  (2016), pp.~783--822.

\bibitem{Br2007}
{\sc D.~Braess}, {\em Finite elements}, Cambridge University Press, Cambridge,
  third~ed., 2007.
\newblock Theory, fast solvers, and applications in elasticity theory,
  Translated from the German by Larry L. Schumaker.

\bibitem{Br2014}
{\sc S.~C. Brenner}, {\em Forty years of the {C}rouzeix-{R}aviart element},
  Numer. Methods Partial Differential Equations, 31 (2015), pp.~367--396.

\bibitem{BS2008}
{\sc S.~C. Brenner and L.~R. Scott}, {\em The mathematical theory of finite
  element methods}, vol.~15 of Texts in Applied Mathematics, Springer, New
  York, third~ed., 2008.

\bibitem{BS1992}
{\sc S.~C. Brenner and L.-Y. Sung}, {\em Linear finite element methods for
  planar linear elasticity}, Math. Comp., 59 (1992), pp.~321--338.

\bibitem{BF1991}
{\sc F.~Brezzi and M.~Fortin}, {\em Mixed and hybrid finite element methods},
  vol.~15 of Springer Series in Computational Mathematics, Springer-Verlag, New
  York, 1991.

\bibitem{BHLLM2018}
{\sc E.~Burman, P.~Hansbo, M.~G. Larson, K.~Larsson, and A.~Massing}, {\em
  Finite element approximation of the laplace--beltrami operator on a surface
  with boundary}, Numerische Mathematik,  (2018).

\bibitem{CB2002}
{\sc C.~Carstensen and S.~Bartels}, {\em Each averaging technique yields
  reliable a posteriori error control in {FEM} on unstructured grids. {I}.
  {L}ow order conforming, nonconforming, and mixed {FEM}}, Math. Comp., 71
  (2002), pp.~945--969.

\bibitem{Ch2008}
{\sc L.~Chen}, {\em Short implementation of bisection in {MATLAB}}, in Recent
  advances in computational sciences, World Sci. Publ., Hackensack, NJ, 2008,
  pp.~318--332.

\bibitem{Ci2002}
{\sc P.~G. Ciarlet}, {\em The finite element method for elliptic problems},
  vol.~40 of Classics in Applied Mathematics, Society for Industrial and
  Applied Mathematics (SIAM), Philadelphia, PA, 2002.
\newblock Reprint of the 1978 original [North-Holland, Amsterdam; MR0520174 (58
  \#25001)].

\bibitem{CD2016}
{\sc B.~Cockburn and A.~Demlow}, {\em Hybridizable discontinuous {G}alerkin and
  mixed finite element methods for elliptic problems on surfaces}, Math. Comp.,
  85 (2016), pp.~2609--2638.

\bibitem{CR1973}
{\sc M.~Crouzeix and P.-A. Raviart}, {\em Conforming and nonconforming finite
  element methods for solving the stationary {S}tokes equations. {I}}, Rev.
  Fran\c{c}aise Automat. Informat. Recherche Op\'{e}rationnelle S\'{e}r. Rouge,
  7 (1973), pp.~33--75.

\bibitem{DM2016}
{\sc A.~Dedner and P.~Madhavan}, {\em Adaptive discontinuous {G}alerkin methods
  on surfaces}, Numer. Math., 132 (2016), pp.~369--398.

\bibitem{DMS2013}
{\sc A.~Dedner, P.~Madhavan, and B.~Stinner}, {\em Analysis of the
  discontinuous {G}alerkin method for elliptic problems on surfaces}, IMA J.
  Numer. Anal., 33 (2013), pp.~952--973.

\bibitem{De2009}
{\sc A.~Demlow}, {\em Higher-order finite element methods and pointwise error
  estimates for elliptic problems on surfaces}, SIAM J. Numer. Anal., 47
  (2009), pp.~805--827.

\bibitem{DD2007}
{\sc A.~Demlow and G.~Dziuk}, {\em An adaptive finite element method for the
  {L}aplace-{B}eltrami operator on implicitly defined surfaces}, SIAM J. Numer.
  Anal., 45 (2007), pp.~421--442 (electronic).

\bibitem{dC1992}
{\sc M.~P. do~Carmo}, {\em Riemannian geometry}, Mathematics: Theory \&
  Applications, Birkh\"auser Boston, Inc., Boston, MA, 1992.
\newblock Translated from the second Portuguese edition by Francis Flaherty.

\bibitem{DG2017}
{\sc G.~Dong and H.~Guo}, {\em Parametric polynomial preserving recovery on
  manifolds}, arXiv:1703.06509 [math.NA], 2017.

\bibitem{Dz1988}
{\sc G.~Dziuk}, {\em Finite elements for the {B}eltrami operator on arbitrary
  surfaces}, in Partial differential equations and calculus of variations,
  vol.~1357 of Lecture Notes in Math., Springer, Berlin, 1988, pp.~142--155.

\bibitem{DE2013}
{\sc G.~Dziuk and C.~M. Elliott}, {\em Finite element methods for surface
  {PDE}s}, Acta Numer., 22 (2013), pp.~289--396.

\bibitem{ETKSD2007}
{\sc S.~Elcott, Y.~Tong, E.~Kanso, P.~Schr\"{o}der, and M.~Desbrun}, {\em
  Stable, circulation-preserving, simplicial fluids}, ACM Trans. Graph., 26
  (2007).

\bibitem{Fr2018}
{\sc T.-P. Fries}, {\em Higher-order surface fem for incompressible
  navier-stokes flows on manifolds}, International Journal for Numerical
  Methods in Fluids, 88 (2018), pp.~55--78.

\bibitem{GR2016}
{\sc J.~Grande and A.~Reusken}, {\em A higher order finite element method for
  partial differential equations on surfaces}, SIAM J. Numer. Anal., 54 (2016),
  pp.~388--414.

\bibitem{GZ2015}
{\sc H.~Guo and Z.~Zhang}, {\em Gradient recovery for the {C}rouzeix-{R}aviart
  element}, J. Sci. Comput., 64 (2015), pp.~456--476.

\bibitem{LL2017}
{\sc K.~Larsson and M.~G. Larson}, {\em A continuous/discontinuous {G}alerkin
  method and a priori error estimates for the biharmonic problem on surfaces},
  Math. Comp., 86 (2017), pp.~2613--2649.

\bibitem{Le2013}
{\sc J.~M. Lee}, {\em Riemannian manifolds}, vol.~176 of Graduate Texts in
  Mathematics, Springer-Verlag, New York, 1997.
\newblock An introduction to curvature.

\bibitem{MCPTD2009}
{\sc P.~Mullen, K.~Crane, D.~Pavlov, Y.~Tong, and M.~Desbrun}, {\em
  Energy-preserving integrators for fluid animation}, ACM Trans. Graph., 28
  (2009), pp.~38:1--38:8.

\bibitem{OQRY2018}
{\sc M.~A. Olshanskii, A.~Quaini, Reusken A., and V.~Yushutin}, {\em A finite
  element method for the surface stokes problem}, arXiv:1801.06589 [math.NA],
  2018.

\bibitem{ORG2009}
{\sc M.~A. Olshanskii, A.~Reusken, and J.~Grande}, {\em A finite element method
  for elliptic equations on surfaces}, SIAM J. Numer. Anal., 47 (2009),
  pp.~3339--3358.

\bibitem{OS2016}
{\sc M.~A. Olshanskii and D.~Safin}, {\em A narrow-band unfitted finite element
  method for elliptic {PDE}s posed on surfaces}, Math. Comp., 85 (2016),
  pp.~1549--1570.

\bibitem{PRPV2017}
{\sc K.~Padberg-Gehle, S.~Reuther, S.~Praetorius, and A.~Voigt}, {\em Transfer
  operator-based extraction of coherent features on surfaces}, in Topological
  Methods in Data Analysis and Visualization IV, H.~Carr, C.~Garth, and
  T.~Weinkauf, eds., Cham, 2017, Springer International Publishing,
  pp.~283--297.

\bibitem{Re2018}
{\sc A.~Reusken}, {\em Stream function formulation of surface stokes
  equations}, IMA Journal of Numerical Analysis,  (2018), p.~dry062.

\bibitem{RV2018}
{\sc S.~Reuther and A.~Voigt}, {\em Solving the incompressible surface
  navier-stokes equation by surface finite elements}, Physics of Fluids, 30
  (2018), p.~012107.

\bibitem{cgal}
{\sc L.~Rineau and M.~Yvinec}, {\em {3D} surface mesh generation}, in {CGAL}
  User and Reference Manual, {CGAL Editorial Board}, {4.9}~ed., 2016.

\bibitem{STY2015}
{\sc E.~Sasaki, S.~Takehiro, and M.~Yamada}, {\em Bifurcation structure of
  two-dimensional viscous zonal flows on a rotating sphere}, Journal of Fluid
  Mechanics, 774 (2015), p.~224–244.

\bibitem{WCH2010}
{\sc H.~Wei, L.~Chen, and Y.~Huang}, {\em Superconvergence and gradient
  recovery of linear finite elements for the {L}aplace-{B}eltrami operator on
  general surfaces}, SIAM J. Numer. Anal., 48 (2010), pp.~1920--1943.

\bibitem{Wl1987}
{\sc J.~Wloka}, {\em Partial differential equations}, Cambridge University
  Press, Cambridge, 1987.
\newblock Translated from the German by C. B. Thomas and M. J. Thomas.

\bibitem{ZN2005}
{\sc Z.~Zhang and A.~Naga}, {\em A new finite element gradient recovery method:
  superconvergence property}, SIAM J. Sci. Comput., 26 (2005), pp.~1192--1213
  (electronic).

\end{thebibliography}
\end{document}